\documentclass[12pt]{amsart}

\usepackage{amssymb,latexsym,amsmath,mathrsfs}
\usepackage{enumerate}

\usepackage{graphicx}

\makeatletter
\@namedef{subjclassname@2010}{%
  \textup{2010} Mathematics Subject Classification}
\makeatother

\newtheorem{thm}{Theorem}[section]
\newtheorem{cor}[thm]{Corollary}
\newtheorem{lem}[thm]{Lemma}
\newtheorem{prop}[thm]{Proposition}

\theoremstyle{definition}
\newtheorem{defin}[thm]{Definition}
\newtheorem*{xrem}{Remark}

\numberwithin{equation}{section}

\newcommand{\Ba}[1]{\begin{array}{#1}}
\newcommand{\Ea}{\end{array}}
\newcommand{\Be}{\begin{equation}}
\newcommand{\Ee}{\end{equation}}
\newcommand{\Bea}{\begin{eqnarray}}
\newcommand{\Eea}{\end{eqnarray}}
\newcommand{\Beas}{\begin{eqnarray*}}
\newcommand{\Eeas}{\end{eqnarray*}}
\newcommand{\Benu}{\begin{enumerate}}
\newcommand{\Eenu}{\end{enumerate}}
\newcommand{\Bi}{\begin{itemize}}
\newcommand{\Ei}{\end{itemize}}
\newcommand{\bprop} {\begin{proposition}}
\newcommand{\eprop} {\end{proposition}}
\newcommand{\bthm} {\begin{theorem}}
\newcommand{\ethm} {\end{theorem}}
\newcommand{\blem} {\begin{lemma}}
\newcommand{\elem} {\end{lemma}}
\newcommand{\bcor} {\begin{corollary}}
\newcommand{\ecor} {\end{corollary}}

\begin{document}

%%%%% To ease editing, for IMPAN journals add:

\baselineskip=17pt

%%%%%%%%%%%

%% In the running head, replace first names by initials 
%% and give an abbreviation of the title.

\title[Carleson embeddings]{Carleson embeddings for Hardy-Orlicz and Bergman-Orlicz spaces of the upper-half plane }

\author[J. M. Tanoh Dje]{Jean Marcel Tanoh Dje}
\address{Laboratoire de Math\'ematiques Fondamentales\\ UFR Math\'ematiques et Informatique\\ Universit\'e F\'elix Houphou\"et-Boigny\\ Abidjan-Cocody, 22 B.P.582 Abidjan 22,C\^ote d' Ivoire}
\email{djetano2017@gmail.com}

\author[B. F. Sehba]{Beno\^it Florent Sehba}
\address{Department of Mathematics, University of Ghana,\\ P. O. Box LG 62 Legon, Accra, Ghana}
\email{bfsehba@ug.edu.gh}

%\date{}

\begin{abstract}
In this paper we characterize off-diagonal Carleson embeddings for both Hardy-Orlicz spaces and Bergman-Orlicz spaces of the upper-half plane. We use these results to obtain embedding relations and pointwise multipliers between these spaces.  
\end{abstract}



\maketitle
\section{Introduction}
\setcounter{equation}{0} \setcounter{footnote}{0}
\setcounter{figure}{0} The aim of this note is to provide Carleson embeddings for some weighted spaces
of holomorphic functions $D^\Phi$ of the upper-half plane. More precisely, we characterize those  positive measures $\mu$ on the upper-half plane such that $D^\Phi$ embeds continuously into $L^\Psi(d\mu)$. Here the space $D^\Phi$ is either a Hardy-Orlicz space or a Bergman-Orlicz space.  Our results are  applied to
the characterization of embedding relations between Hardy-Orlicz spaces and Bergman-Orlicz spaces or just between Bergman-Orlicz spaces. We also characterize pointwise  multipliers from Hardy-Orlicz spaces or Bergman-Orlicz spaces to Bergman-Orlicz spaces.
\vskip .3cm
Recall that the upper-half plane is the subset $\mathbb{C}_+$ of the complex plane $\mathbb{C}$ defined by $$\mathbb{C}_+:=\{x+iy\in \mathbb{C}:y>0\}.$$ 
A continuous and nondecreasing
function $\Phi$ from $[0,\infty)$ onto itself is called a growth function. Note that this implies that $\Phi(0)=0$.
\vskip .1cm
For $\Phi$ a growth function, the Luxembourg (quasi)-norm on $L^\Phi(\mathbb{R})$ is the quantity $$\|f\|_{L^\Phi}^{lux}:=\inf\{\lambda>0:\,\,\int_{\mathbb{R}}\Phi\left(\frac{|f(x)|}{\lambda}\right)dx\le 1\}.$$

Given $\Phi$ a growth function, the Hardy-Orlicz space $H^{\Phi}(\mathbb{C}_+)$ is the space of all holomorphic functions $f$ on $\mathbb{C_+}$ such that $$\|f\|_{H^{\Phi}}^{lux}:=\sup_{y>0}\|f(\cdot+iy)\|_{L^\Phi}^{lux}<\infty.$$
For $\alpha>-1$, we write $dV_\alpha(z)=y^\alpha dV(z)$ where $dV(x+iy)=dxdy$. For $\Phi$ a growth function and $\alpha>-1$, the Bergman-Orlicz space $A_\alpha^{\Phi}(\mathbb{C}_+)$ is the subspace of the Orlicz space $L_\alpha^\Phi(\mathbb{C}_+)$ consisting of all holomorphic functions on $\mathbb{C_+}$. Recall that $f\in L_\alpha^\Phi(\mathbb{C}_+)$ if $$\|f\|_{L_\alpha^{\Phi}}=\|f\|_{\Phi,\alpha}:=\int_{\mathbb{C}_+}\Phi(|f(z)|)dV_\alpha(z)<\infty.$$
We also endow $A_\alpha^\Phi(\mathbb{C}_+)$ with the following Luxembourg (quasi)-norm  defined on $L_\alpha^\Phi(\mathbb{C}_+)$ by $$\|f\|_{\Phi,\alpha}^{lux}:=\inf\left\{\lambda>0:\,\,\int_{\mathbb{C}_+}\Phi\left(\frac{|f(z)|}{\lambda}\right)dV_\alpha(z)\le 1\right\}.$$ 
We note that when $\Phi(t)=t^p$, $0<p<\infty$, $H^\Phi(\mathbb{C}_+)$ and  $A_\alpha^\Phi(\mathbb{C}_+)$ are just the usual Hardy space and Bergman space $H^p(\mathbb{C}_+)$ and $A_\alpha^p(\mathbb{C}_+)$ respectively defined as the spaces of all holomorphic functions $f$ on $\mathbb{C}_+$ such that $$\|f\|_{H^p}^p:=\sup_{y>0}\int_{\mathbb{R}}|f(x+iy)|^pdx<\infty$$
and
$$\|f\|_{A_\alpha^p}^p:=\int_{\mathbb{C}_+}|f(z)|^pdV_\alpha(z)<\infty.$$
A growth function $\Phi$ is said to be of upper type $q$ if we can find $q > 0$ and $C>0$ such that, for $s>0$ and $t\ge 1$,
\begin{equation}\label{uppertype}
 \Phi(st)\le Ct^q\Phi(s).\end{equation}
%We say that $\Phi$ is of lower type $p$ (resp. upper type $q$) when (\ref{lowertype}) (resp. (\ref{uppertype})) is satisfied.
We denote by $\mathscr{U}^q$ the set of growth functions $\Phi$ of upper type $q$, (with $q\ge 1$), such that the function $t\mapsto \frac{\Phi(t)}{t}$ is non-decreasing. We write $$\mathscr{U}=\bigcup_{q\geq 1}\mathscr{U}^q.$$
Note that we may always suppose that any $\Phi\in\mathscr{U}$ is convex and
that $\Phi$ is a $\mathscr{C}^1$ function with derivative $\Phi'(t)\backsimeq \frac{\Phi(t)}{t}$.
\vskip .3cm
For $\Phi_1,\Phi_2\in \mathscr{U}$, our main concern in this note is the characterization of all positive measures $\mu$ on $\mathbb{C}_+$ such that $H^{\Phi_1}(\mathbb{C}_+)$ (resp. $A_\alpha^{\Phi_1}(\mathbb{C}_+)$) embeds continuously into $L^{\Phi_2}(\mathbb{C}_+, d\mu)$.
\vskip .2cm
In the case of the unit disc, the continuous embedding $H^p\hookrightarrow L^q(d\mu)$ was first considered by L. Carleson \cite{carleson1,carleson2} for $p=q$. The case $0<p\le q<\infty$ for the unit disc was solved by P. Duren in \cite{duren}. Since then the problem has been considered by several authors for both Hardy and Bergman spaces of various domains for $\Phi_1(t)=t^p$ and $\Phi_2(t)=t^q$, $0<p,q<\infty$ (see \cite{CW,hastings,hormander,luecking1,luecking2,luecking3,power,Ueki} and the references therein). In the unit ball, the continuous embeddings $H^{\Phi_1}\hookrightarrow L^{\Phi_2}(d\mu)$ and  $A_\alpha^{\Phi_1}\hookrightarrow L^{\Phi_2}(d\mu)$ for $\frac{\Phi_2}{\Phi_1}$ nondecreasing were obtained in \cite{Charpentier,Charpentiersehba,sehba}.
\vskip .1cm
The characterization of the measures $\mu$ for which the embedding $H^p(\mathbb{C}_+)\hookrightarrow L^q(\mathbb{C}_+ d\mu)$ holds, essentially makes use of techniques from harmonic analysis (for $p=q$, see for example \cite[Ch. 7]{grafakos}). One of the further main difficulties when working with growth functions, is the fact that they are not multiplicative (i.e. $\Phi(ab)\neq\Phi(a)\Phi(b)$) in general. Hence to handle Carleson measures here, we develop an approach also based on techniques of harmonic analysis that allows us to overcome the mentioned obstacle and extend the classical results.
%{\bf Sehba: elaborer un peu plus, plus de details sur qui a fait quoi}
\vskip .1cm
Carleson embeddings are very useful in the study of various questions in analytic function spaces: continuous inclusion between spaces, pointwise multipliers, composition operators, integration operators to name a few (see for example \cite{Att,Axler1,Charpentier,Charpentiersehba,luecking,sehbastevic,Ueki,Vukotic,Zhao} and the references therein). These applications are our main motivation for considering these questions here.
\section{Presentation of the results}
We present in this section our main results and some applications.
\subsection{Carleson embeddings for $H^\Phi(\mathbb{C}_+)$ and $A_\alpha^\Phi(\mathbb{C}_+)$}
The complementary function $\Psi$ of the convex growth function $\Phi$, is the function defined from $\mathbb R_+$ onto itself by
\begin{equation}\label{complementarydefinition}
\Psi(s)=\sup_{t\in\mathbb R_+}\{ts - \Phi(t)\}.
\end{equation}
%We observe that if $\Phi\in \mathscr{U}^q$, then $\Psi$ is a growth function of lower type such that the function which $t\mapsto \frac{\Psi(t)}{t}$ is non-decreasing. 

The growth function $\Phi$ satisfies the $\Delta_2$-condition if there exists a constant $K>1$ such that, for any $t\ge 0$,
\begin{equation}\label{eq:delta2condition}
 \Phi(2t)\le K\Phi(t).\end{equation}
It follows easily from (\ref{uppertype}) that any growth function $\Phi\in \mathscr{U}$ satisfies the $\Delta_2$-condition.
We say that the growth function $\Phi$ satisfies the $\bigtriangledown_2-$condition whenever both $\Phi$ and its complementary function satisfy the $\Delta_2-$conditon.
\vskip .3cm
For any interval $I\subset \mathbb{R}$, we recall that the Carleson square above $I$ is the set 
$$Q_I:=\{z=x+iy\in \mathbb{C}: x\in I\,\,\,\textrm{and}\,\,\,0<y<|I|\}.$$
The following definition is adapted from \cite{sehba}.
\begin{defin}
Let $\Phi$ be a growth function. A positive Borel measure $\mu$ on $\mathbb{C}_+$ is called a $\Phi$-Carleson measure, if there is a constant $C>0$ such that for any finite interval $I\subset \mathbb{R}$, 
\Be\label{eq:phicarldef}
\mu(Q_I)\le \frac{C}{\Phi\left(\frac 1{|I|}\right)}.
\Ee
\end{defin}
Our first Carleson embedding result is as follows.
\begin{thm}\label{thm:main1}
Let $\Phi_1$ and $\Phi_2$ be two $\mathcal{C}^1$ convex growth functions with $\Phi_2\in\mathscr{U}$. Assume that $\Phi_1$ satisfies the $\nabla_2$-condition and that $\frac{\Phi_2}{\Phi_1}$ is nondecreasing. Let $\mu$ be a positive Borel measure on $\mathbb{C}_+$. Then the following assertions are equivalent.
\begin{itemize}
\item[(a)] $\mu$ is a $\Phi_2\circ\Phi_1^{-1}$-Carleson measure.
\item[(b)] There exists a constant $C>0$ such that 
\Be\label{eq:equivcarlhardy1}
\sup_{z=x+iy\in \mathbb{C}_+}\int_{\mathbb{C}_+}\Phi_2\left(\Phi_1^{-1}\left(\frac{1}{y}\right)\frac{y^2}{|z-\bar{w}|^2}\right)d\mu(w)\le C<\infty.
\Ee
\item[(c)] There exists a constant $K>0$ such that for any $f\in H^{\Phi_1}(\mathbb{C}_+)$, $f\neq 0$,
\Be\label{eq:equivcarlhardy2}
\int_{\mathbb{C}_+}\Phi_2\left(\frac{|f(z)|}{K\|f\|_{H^{\Phi_1}}^{lux}}\right)d\mu(z)<\infty.
\Ee
\end{itemize}
\end{thm}
Note that the equivalence (a)$\Leftrightarrow$(b) holds even without the additional assumption ``$\Phi_1$ satisfies the $\nabla_2$-condition''. This assumption is needed only in the proof of the assertion (c) and this is due to our method of proof which involves the Hardy-Littlewood maximal function whose boundedness on Orlicz spaces is known only under our assumption (see \cite[Theorem 1.2.1]{koki}).
\vskip .1cm
\begin{defin}
Let $\Phi$ be a growth function and let $\alpha>-1$. A positive Borel measure $\mu$ on $\mathbb{C}_+$ is called a $(\Phi,\alpha)$-Carleson measure, if there is a constant $C>0$ such that for any finite interval $I\subset \mathbb{R}$, 
\Be\label{eq:phicarlbergdef}
\mu(Q_I)\le \frac{C}{\Phi\left(\frac 1{|I|^{2+\alpha}}\right)}.
\Ee
\end{defin}

We obtain the following Carleson embedding result for weighted Bergman-Orlicz spaces.
\begin{thm}\label{thm:main2}
Let $\Phi_1$ and $\Phi_2$ be two $\mathcal{C}^1$ convex growth functions with $\Phi_2\in\mathscr{U}$. Assume that $\Phi_1$ satisfies the $\nabla_2$-condition and that $\frac{\Phi_2}{\Phi_1}$ is nondecreasing. Let $\mu$ be a positive Borel measure on $\mathbb{C}_+$ and let $\alpha>-1$. Then the following assertions are equivalent.
\begin{itemize}
\item[(a)] There exists a constant $C_1>0$ such that for any interval $I\subset \mathbb{R}$,
\Be\label{eq:equivcarlberg1}
 \mu(Q_I)\le \frac{C_1}{\Phi_2\circ\Phi_1^{-1}\left(\frac{1}{|I|^{2+\alpha}}\right)}.
 \Ee
\item[(b)] There exists a constant $C_2>0$ such that 
\Be\label{eq:equivcarlberg2}
\sup_{z=x+iy\in \mathbb{C}_+}\int_{\mathbb{C}_+}\Phi_2\left(\Phi_1^{-1}\left(\frac{1}{y^{2+\alpha}}\right)\frac{y^{4+2\alpha}}{|z-\bar{w}|^{4+2\alpha}}\right)d\mu(w)\le C_2<\infty.
\Ee
\item[(c)] There exists a constant $C_3>0$ such that for any $f\in A_\alpha^{\Phi_1}(\mathbb{C}_+)$, $f\neq 0$,
\Be\label{eq:equivcarlberg3}
\int_{\mathbb{C}_+}\Phi_2\left(\frac{|f(z)|}{C_3\|f\|_{A_\alpha^{\Phi_1}}^{lux}}\right)d\mu(z)<\infty.
\Ee
\end{itemize}
\end{thm}
%\vskip .3cm
\subsection{Application to some inclusion relations}
We apply the above results in giving exact conditions under which a Hardy-Orlicz space or a Bergman-Orlicz space as given above embeds continuously into another Bergman-Orlicz space. In the unit disc of $\mathbb{C}$ or the unit ball of $\mathbb{C}^n$, for the classical Hardy and Bergman spaces, these characterizations are well known (see \cite{ZZ,Zhu} and the references therein). Embedding relations between Bergman-Orlicz spaces of the unit ball have been obtained by the second author in \cite{sehba}.
\vskip .2cm
We first have the following result.
\begin{thm}\label{thm:embed1}
Let $\Phi_1$ and $\Phi_2$ be two $\mathcal{C}^1$ convex growth functions with $\Phi_2\in\mathscr{U}$, and let $\alpha>-1$. Assume that $\Phi_1$ satisfies the $\nabla_2$-condition and that $\frac{\Phi_2}{\Phi_1}$ is nondecreasing. Then the Hardy-Orlicz space $H^{\Phi_1}(\mathbb{C}_+)$ embeds continuously into the Bergman-Orlicz space $A_\alpha^{\Phi_2}(\mathbb{C}_+)$ if and only if there exists a constant $C>0$ such that for any $t\in (0,\infty)$,
\Be\label{eq:embedhardybergcond}
\Phi_1^{-1}(t)\leq \Phi_2^{-1}(Ct^{2+\alpha}).
\Ee
\end{thm}
\vskip .1cm
We remark that in the case $\Phi_1(t)=t^p$ and $\Phi_2(t)=t^q$ with $0<p<q<\infty$, the condition (\ref{eq:embedhardybergcond}) reduces to $\frac 1p=\frac{2+\alpha}{q}$.
\vskip .1cm
We also obtain the following.
\begin{thm}\label{thm:embed2}
Let $\Phi_1$ and $\Phi_2$ be two $\mathcal{C}^1$ convex growth functions with $\Phi_2\in\mathscr{U}$, and let $\alpha,\beta>-1$. Assume that $\Phi_1$ satisfies the $\nabla_2$-condition and that $\frac{\Phi_2}{\Phi_1}$ is nondecreasing. Then the Bergman-Orlicz space $A_\alpha^{\Phi_1}(\mathbb{C}_+)$ embeds continuously into the Bergman-Orlicz space $A_\beta^{\Phi_2}(\mathbb{C}_+)$ if and only if there exists a constant $C>0$ such that for any $t\in (0,\infty)$,
\Be\label{eq:embedbergbergcond}
\Phi_1^{-1}(t^{2+\alpha})\leq \Phi_2^{-1}(Ct^{2+\beta}).
\Ee
\end{thm}
\vskip .1cm
It is easy to see that in the case $\Phi_1(t)=t^p$ and $\Phi_2(t)=t^q$ with $0<p<q<\infty$, the condition (\ref{eq:embedbergbergcond}) reduces to $\frac {2+\alpha}p=\frac{2+\beta}{q}$.

\subsection{Application to pointwise multipliers}
Let $X$ and $Y$ be two analytic function spaces which are metric spaces, with respective metrics $d_X$ and $d_Y$. An analytic function $g$ is said to be a multiplier from $X$ to $Y$, if there exists a constant $C>0$ such that for any $f\in X$, $$d_Y(fg,0)\leq Cd_X(f,0).$$
We denote by $\mathcal{M}(X,Y)$ the set of multipliers from $X$ to $Y$. 

Multipliers between usual Bergman spaces of the unit disc and the unit ball have been obtained in \cite{Att, Axler1, Axler2, luecking, Vukotic, Zhao}. In \cite{sehba}, the first author, using Carleson embeddings for Bergman-Orlicz spaces of the unit ball $\mathbb{B}^n$ of $\mathbb{C}^n$ characterized pointwise multipliers from $A_\alpha^{\Phi_1}(\mathbb{B}^n)$ to $A_\beta^{\Phi_2}(\mathbb{B}^n)$ where $\Phi_1$ and $\Phi_2$ are growth functions such that $\frac{\Phi_2}{\Phi_1}$ is nondecreasing and $\Phi_2$ is in some subclass $\tilde{\mathscr U}$ of $\mathscr U$. We provide here the same type of results for Hardy-Orlicz and Bergman-Orlicz spaces of the upper-half plane.

\vskip .2cm
We say a growth function $\Phi\in \mathscr U^q$ belongs to $\tilde{\mathscr U}$, if the following three conditions are satisfied.
\begin{itemize}
\item[($a_1$)] There exists a constant $C_1>0$ such that for any $0<s,t<\infty$,
\begin{equation}\label{eq:uppertypecondmulti1}
\Phi(st)\leq C_1\Phi(s)\Phi(t).
\end{equation}
\item[($a_2$)]  There exists a constant $C_2>0$ such that for any $a,b\geq 1$, 
\begin{equation}\label{eq:uppertypecondmulti2}
\Phi\left(\frac{a}{b}\right)\leq C_2\frac{\Phi(a)}{b^q}.
\end{equation}
\item[($a_3$)]  There exists a constant $C_3>0$ such that for any $0<a\le b\le 1$, 
\begin{equation}\label{eq:uppertypecondmulti3}
\Phi\left(\frac{a}{b}\right)\leq C_3\frac{\Phi(a)}{\Phi(b)}.
\end{equation}

\end{itemize}
Clearly, power functions are in $\tilde{\mathscr U}$. As nontrivial member of $\tilde{\mathscr U}$, we have the function $t\mapsto t^q\log^\alpha(C+t)$, where $q\geq 1$, $\alpha>0$ and the constant $C>0$ is large enough. 
\vskip .2cm
Let $\omega:(0,\infty)\longrightarrow (0,\infty)$ be a continuous function. An analytic function $f$ in $\mathbb C_+$ is said to be in $\mathcal H_\omega^\infty(\mathbb C_+)$ if
\begin{equation}\label{OmegaHinfdef}
||f||_{\mathcal H_\omega^\infty}:=\sup_{z\in \mathbb C_+}\frac{|f(z)|}{\omega(\Im z)}<\infty.
\end{equation}
We observe that $\mathcal H_\omega^\infty(\mathbb C_+)$ is a Banach space. 

The following result provides pointwise multipliers from Hardy-Orlicz spaces to Bergman-Orlicz spaces.
\begin{thm}\label{thm:main3}
Let $\Phi_1\in \mathscr U$ and $\Phi_2\in \tilde{\mathscr U}$. Assume that $\frac{\Phi_2}{\Phi_1}$ is non-decreasing. Let $\alpha>-1$ and define for $t\in (0,\infty)$, the function
$$\omega(t)=\frac{\Phi_2^{-1}\left(\frac{1}{t^{2+\alpha}}\right)}{\Phi_1^{-1}\left(\frac{1}{t}\right)}.$$
Then the following assertions hold.
\begin{itemize}
\item[(i)] If $\Phi_1$ satisfies the $\nabla_2$-condition, and $\omega$ is equivalent to $1$, then $$\mathcal{M}\left(H^{\Phi_1}(\mathbb{C}_+),A_\alpha^{\Phi_2}(\mathbb{C}_+)\right)=H^\infty(\mathbb{C}_+).$$
\item[(ii)] If $\omega$ is non-decreasing on $(0,\infty)$ and $\lim_{t\rightarrow 0}\omega(t)=0$, then $$\mathcal{M}\left(H^{\Phi_1}(\mathbb{C}_+),A_\alpha^{\Phi_2}(\mathbb{C}_+)\right)=\{0\}.$$
\item[(iii)] If $\Phi_1$ and $\Phi_2\circ\Phi_1^{-1}$ satisfies the $\nabla_2$-condition, and $\omega$ is non-increasing on $(0,\infty)$, then $$\mathcal{M}\left(H^{\Phi_1}(\mathbb{C}_+),A_\alpha^{\Phi_2}(\mathbb{C}_+)\right)=H_\omega^\infty(\mathbb{C}_+).$$

\end{itemize}
\end{thm}
The next result provides pointwise multipliers between two different Bergman-Orlicz spaces.
\begin{thm}\label{thm:main4}
Let $\Phi_1\in \mathscr U$ and $\Phi_2\in  \tilde{\mathscr U}$. Assume that  $\frac{\Phi_2}{\Phi_1}$ is nondecreasing. Let $\alpha, \beta>-1$ and define for $t\in (0,\infty)$, the function
$$\omega(t)=\frac{\Phi_2^{-1}\left(\frac{1}{t^{2+\beta}}\right)}{\Phi_1^{-1}\left(\frac{1}{t^{2+\alpha}}\right)}.$$
Then the following assertions hold.
\begin{itemize}
\item[(i)] If $\Phi_1$ satisfies the $\nabla_2$-condition, and $\omega$ is equivalent to $1$, then $$\mathcal{M}\left(A_\alpha^{\Phi_1}(\mathbb{C}_+),A_\beta^{\Phi_2}(\mathbb{C}_+)\right)=H^\infty(\mathbb{C}_+).$$
\item[(ii)] If $\omega$ is non-decreasing on $(0,\infty)$ and $\lim_{t\rightarrow 0}\omega(t)=0$, then $$\mathcal{M}\left(A_\alpha^{\Phi_1}(\mathbb{C}_+),A_\beta^{\Phi_2}(\mathbb{C}_+)\right)=\{0\}.$$
\item[(iii)] If $\Phi_1$ and $\Phi_2\circ\Phi_1^{-1}$ satisfy the $\nabla_2$-condition, and $\omega$ is non-increasing on $(0,\infty)$, then $$\mathcal{M}\left(A_\alpha^{\Phi_1}(\mathbb{C}_+),A_\beta^{\Phi_2}(\mathbb{C}_+)\right)=H_\omega^\infty(\mathbb{C}_+).$$

\end{itemize}
\end{thm}
\vskip .1cm
In the above two results, we require $\Phi_1$ to satisfy the $\nabla_2$-condition because we aim to apply Theorem \ref{thm:main1} and Theorem \ref{thm:main2} where this hypothesis is used. In assertion (iii) of these results, we also require $\Phi_2\circ\Phi_1^{-1}$ to satisfy the $\nabla_2$-condition. This is needed to prove that the measure $$d\mu(x+iy)=\frac{dxdy}{y^2\Phi_2\circ\Phi_1^{-1}\left(\frac 1{y^{2+\alpha}}\right)}$$ appearing in our proofs is a  $\left(\Phi_2\circ\Phi_1^{-1},\alpha\right)$-Carleson measure. In the case where this condition does not hold, it is easy to exhibit an example of $\Phi_1$ and $\Phi_2$ for which the measure $\mu$ is not a  $\left(\Phi_2\circ\Phi_1^{-1},\alpha\right)$-Carleson measure.
\vskip .3cm
In the next section, we introduce more definitions and present some results that we need in our presentation. In Section 4, we present the proofs of the Carleson embeddings results; in Section 5, we prove the results on the continuous inclusion of a Hardy-Orlicz or Bergman-Orlicz space into another Bergman-Orlicz space. Section 6 contains the proofs of the pointwise multipliers results. In the last section, we conclude our presentation, taking advantage of this part to present the corresponding weak-type results.
\vskip .3cm
As usual, given two positive quantities $A$ and $B$, the notation $A\lesssim B$
means that for some positive constant $C$, $A\le CB$. When $A\lesssim B$ and $B\lesssim A$, we write $A\thickapprox B$. In general $C$ or $C_s$, $s\in \mathbb{R}$ will denote a constant (depending only on the underlined variable) whose value is not necessarily the same for different occurrences. 
\section{Some useful facts}
We present in this section some useful results needed in our presentation.
\subsection{Some properties of growth functions}
We recall that a growth function $\Phi$ is of lower type $p$ if we can find $p > 0$ and $C>0$ such that, for $s>0$ and $0<t\le 1$,
\begin{equation}\label{lowertype}
 \Phi(st)\le Ct^p\Phi(s).\end{equation}
%We say that $\Phi$ is of lower type $p$ (resp. upper type $q$) when (\ref{lowertype}) (resp. (\ref{uppertype})) is satisfied.
We denote by $\mathscr{L}_p$ the set of growth functions $\Phi$ of lower type $p$,  $0<p\le 1$, such that the function $t\mapsto \frac{\Phi(t)}{t}$ is non-increasing. We write $$\mathscr{L}=\bigcup_{0<p\leq 1}\mathscr{L}_p.$$
We recall with \cite[Proposition 2.1]{sehbatchoundja} that $\Phi\in \mathscr{L}_p$ if and only if $\Phi^{-1}\in \mathscr{U}^{1/p}$.
\vskip .2cm
We recall that for $\Phi$ a $\mathcal C^1$ growth function, the lower and the upper indices of $\Phi$ are respectively defined by
$$a_\Phi:=\inf_{t>0}\frac{t\Phi^\prime(t)}{\Phi(t)}\,\,\,\textrm{and}\,\,\,b_\Phi:=\sup_{t>0}\frac{t\Phi^\prime(t)}{\Phi(t)}.$$
We also recall that if $\Phi$ is convex, then $1\le a_\Phi\le b_\Phi<\infty$. Following \cite[Lemma 2.6]{DHZZ} we have that a convex growth function satisfies the $\nabla_2-$condition if and only if $1< a_\Phi\le b_\Phi<\infty$. 
Let us observe that if $\Phi$ is a $\mathcal C^1$ growth function, then the function $\frac{\Phi(t)}{t^{a_\Phi}}$ is increasing while the function $\frac{\Phi(t)}{t^{b_\Phi}}$ is decreasing. These observations imply in particular that if $\Phi$ is $\mathcal{C}^1$ convex growth function that satisfies the  $\nabla_2-$condition, then $\Phi\in \mathscr{U}$.

The following will be useful.
\begin{lem}\label{lem:dini}
Let $\Phi$ be a convex growth function that satisfies the $\Delta_2$-condition. Then the following assertions are equivalent.
\begin{itemize}
\item[(a)] $\Phi$ satisfies the $\nabla_2$-condtion. 
\item[(b)] There is a constant $C_1>0$ such that for any $t>0$,
\Be\label{eq:dini}
\int_0^t\frac{\Phi(s)}{s^2}ds\le C_1\frac{\Phi(t)}{t}.
\Ee
\item[(c)] There is a constant $C_2>1$ such that for any $t>0$, $\Phi(C_2t)\ge 2C_2\Phi(t)$.
\end{itemize}
\end{lem}
\begin{proof}
We prove that (a)$\Rightarrow$(b)$\Rightarrow$(c)$\Rightarrow$(a).
\vskip .1cm
 (a)$\Rightarrow$(b): Assume that $\Phi$ satisfies the $\nabla_2$-condtion. We start by observing that
\Beas
\int_0^t\frac{\Phi(s)}{s^2}ds &=& \sum_{j=0}^\infty\int_{2^{-j-1}t}^{2^{-j}t}\frac{\Phi(s)}{s^2}ds\\ &\le& \sum_{j=0}^\infty\frac{\Phi(2^{-j}t)}{2^{-2(j+1)}t^2}2^{-j-1}t.
\Eeas
Let $p$ be the lower indice of $\Phi$. As $\Phi$ satisfies the $\nabla_2$-condtion, we have that $p>1$. As $t\rightarrow \frac{\Phi(t)}{t^p}$ is increasing, we obtain that for $j\ge 0$, $\Phi(2^{-j}t)\le 2^{-jp}\Phi(t)$. Hence
\Beas
\int_0^t\frac{\Phi(s)}{s^2}ds  &\le& 2\frac{\Phi(t)}{t}\sum_{j=0}^\infty 2^{-j(p-1)}\\ &\lesssim& \frac{\Phi(t)}{t}.
\Eeas
\vskip .1cm
(b)$\Rightarrow$(c): Assume that $\Phi$ satisfies (\ref{eq:dini}), i.e. 
 $~\forall~t>0 ~$,
\[  \int_0^t\frac{\Phi(s)}{s^2}ds\le C_1\frac{\Phi(t)}{t}. \]
Let $d>2$ be fixed. As the function
 $t\rightarrow \dfrac{\Phi(t)}{t}$ is nondecreasing, we have that
 \[ \int_0^t\frac{\Phi(s)}{s^2}ds \geq \int_\frac{t}{d}^\frac{t}{2}\frac{\Phi(s)}{s}\frac{ds}{s} \geq \frac{\Phi(\frac{t}{d})}{\frac{t}{d}}\int_\frac{t}{d}^\frac{t}{2} \frac{ds}{s}= \frac{\Phi(\frac{t}{d})}{\frac{t}{d}}\ln(\frac{d}{2})    \]
Hence \[d\Phi(\frac{t}{d}) \leq \frac{C_1}{\ln(\frac{d}{2})}\Phi(t).\]   
Let us choose $d\geq 2e^{2C_1}$ such that $\dfrac{C_1}{\ln(\frac{d}{2})} \leq \dfrac{1}{2}$. Then
\[   d\Phi(\frac{t}{d}) \leq \frac{1}{2}\Phi(t). \]   
That is \[ 2d\Phi(u) \leq \Phi(du).      \] 
\vskip .1cm
(c)$\Rightarrow$(a): Assume that there exists $C_2>1$ such that $~\forall~t>0$, $\Phi(C_2t)\ge 2C_2\Phi(t)$. We only have to prove that the complementary function $\Psi$ of $\Phi$ satisfies the $\Delta_2$-condition. 
\vskip .1cm 
Let $~t>0$. Put \[ \Phi_{1}(t)= \frac{1}{2C_2}\Phi(C_2t).     \]
Then $\Phi_{1}$ belongs to $\mathscr{U}$. Let $\Psi_{1}$ be the complementary function of $\Phi_{1}$. We have that
for any $u \geq 0$, \[\Psi_{1}(u) =\sup_{t\geq 0}\{ut- \Phi_{1}(t)  \}= \frac{1}{2C_2}\Psi(2u).    \]
Hence
$$\begin{array}{rcl}
\Phi(C_2t)\ge 2C_2\Phi(t) &\Leftrightarrow &
\Phi(t) \leq \Phi_{1}(t) \\\\
&\Rightarrow&
\Psi_{1}(u) \leq \Psi(u) \\\\
&\Rightarrow&
\Psi(2u) \leq 2C_2 \Psi(u).
\end{array}$$ 
Thus $\Psi$ satisfies the $\Delta_{2}$-condition.
As $\Phi$ and its complementary function $\Psi$ satisfy the $\Delta_{2}$-condition, we conclude that $\Phi$ satisfies the $\nabla_2$-condtion.
\end{proof}
\begin{lem}\label{lem:reverseprodphi}
Let $\Phi_1,\Phi_2$ be two convex growth functions.  Assume that $\Phi_2\in \mathscr{U}^q$ and that $\frac{\Phi_2}{\Phi_1}$ is nondecreasing. Then the function $\Phi_3$ defined by $\Phi_3(0)=0$ and $$\Phi_3(t)=\frac{1}{\Phi_2\circ\Phi_1^{-1}\left(\frac 1t\right)},\,\,\,\textrm{for}\,\,\,t>0$$
belongs to the class $\mathscr{U}$.
\end{lem}
\begin{proof}
Note that as $\frac{\Phi_1^{-1}(t)}{t}$ is nonincreasing, we have that for any $s\ge 1$ and $t>0$, $$\Phi_1^{-1}(st)\le s\Phi_1^{-1}(t)$$ and so 
\Beas\Phi_2\circ\Phi_1^{-1}(st) &\le& \Phi_2\left(s\Phi_1^{-1}(t)\right)\\ &\le& Cs^q\Phi_2\circ\Phi_1^{-1}(t).\Eeas
That is $\Phi_2\circ\Phi_1^{-1}\in \mathscr{U}^q$.
Hence $$\frac{\Phi_2\circ\Phi_1^{-1}\left(\frac 1{st}\right)}{\left(\frac 1{st}\right)^q}\ge \frac{\Phi_2\circ\Phi_1^{-1}\left(\frac 1{t}\right)}{\left(\frac 1{t}\right)^q},$$
or equivalently, $$\Phi_2\circ\Phi_1^{-1}\left(\frac 1{st}\right)\ge \frac{\Phi_2\circ\Phi_1^{-1}\left(\frac 1{t}\right)}{s^q}.$$
That is for any $s\ge 1$ and $t>0$, $\Phi_3(st)\le s^q\Phi_3(t)$. Lemma follows easily as $\Phi_3$ and the function $t\mapsto \frac{\Phi_3(t)}{t}$ are increasing.
\end{proof}
\subsection{Integrability results for some positive kernel functions}
We recall that the beta function is defined by $$B(m, n)= B(n,m) = \int_0^{\infty}\!\!\!\frac{u^{m-1}}{(1+u)^{m+n}}\mathrm{d}u\qquad\mbox{where}\quad m,n>0.$$
The two following results can be found for example in \cite{BanSeh}.
\begin{lem}\label{lem:betafunctionconditions}
Let $\alpha, \beta$ be a real numbers, and $t>0$ be fixed. Then the integral $$I(t)=\int_0^\infty \frac{y^\alpha}{(t+y)^\beta}\mathrm{d}x$$
converges if and only if $\alpha>-1$ and $\beta-\alpha>1$. In this case, $$I(y)=B(\alpha+1, \beta-\alpha-1)t^{-\beta+\alpha+1}.$$
\end{lem}
\vskip .1cm
%We will need the following integrability conditions of the kernel function.
\begin{lem}\label{lem:integkernel} Let $\alpha$ be real. Then for $y>0$ fixed, the integral
$$J_{\alpha}(y)=\int_{\mathbb{R}}\frac{dx}{|x+iy|^\alpha}
$$ converges if and only if $\alpha > 1.$ In this case,
$$J_{\alpha}(y)=B(\frac{1}{2}, \frac{\alpha-1}{2})y^{1-\alpha}.$$
\end{lem}
\subsection{Hardy-Orlicz spaces of the upper-half plane}
For $\Phi\in \mathscr{U}^q$ and $f\in H^\Phi(\mathbb{C}_+)$, we define $$\|f\|_{H^\Phi}:=\sup_{y>0}\int_{\mathbb{R}}\Phi\left(|f(x+iy)|\right)dx.$$
One can check that $f\in H^\Phi(\mathbb{C}_+)$ if and only if $\|f\|_{H^\Phi}<\infty$. Indeed, we have that the following relations hold:
$$\|f\|_{L^\Phi}\lesssim \max\{\|f\|_{L^\Phi}^{lux},\left(\|f\|_{L^\Phi}^{lux}\right)^q\}$$
and $$\|f\|_{L^\Phi}^{lux}\lesssim \max\{\|f\|_{L^\Phi},\left(\|f\|_{L^\Phi}\right)^{1/q}\}.$$
%{\bf Sehba: Quelqu'un doit verifier ces deux inegalites.}
\vskip .1cm
Also $\|\cdot\|_{H^\Phi}^{lux}$ defines a norm on $H^\Phi(\mathbb{C}_+)$ and $(H^\Phi(\mathbb{C}_+),\|\cdot\|_{H^\Phi}^{lux})$ is a Banach space. 
\vskip .2cm
Let us observe the following.
\begin{lem}
Let $\Phi$ a convex growth function. Then $\|f\|_{H^\Phi}=0$ if and only if $f=0$.
\end{lem}
\begin{proof}
Assume that $\|f\|_{H^\Phi}=0$. Then for any $y>0$ fixed, there exists $\delta_0=\delta_0(y)>0$ such that  for any $0<\delta<\delta_0$,
$$\int_{\mathbb{R}}\Phi\left(\frac{|f(x+iy)|}{\delta}\right)dx\le 1.$$
This implies that for any interval $I\subset \mathbb{R}$,
$$\int_{I}\Phi\left(\frac{|f(x+iy)|}{\delta}\right)\frac{dx}{|I|}\le \frac{1}{|I|}.$$
We obtain in particular that for any $C>1$,
$$\Phi\left(\int_{I}\frac{C|f(x+iy)|}{\delta_0}\frac{dx}{|I|}\right)\le \int_{I}\Phi\left(\frac{|f(x+iy)|}{(\delta_0/C)}\right)\frac{dx}{|I|}\le \frac{1}{|I|}.$$
Thus 
$$\int_{I}\frac{|f(x+iy)|}{\delta_0}dx\le \frac{|I|}{C}\Phi^{-1}\left(\frac{1}{|I|}\right).$$
Letting $C\rightarrow \infty$, we obtain that for any interval $I\subset \mathbb{R}$,
$$\int_{I}|f(x+iy)|dx=0.$$
Hence the Monotone Convergence Theorem then gives that
$$\int_{\mathbb{R}}|f(x+iy)|dx=0.$$
Thus $f=0$. The proof is complete.

\end{proof}
We recall that the Hardy-Littlewood maximal function of $\mathbb{R}$ is the function defined for any locally integrable function $f$ by
\Be\label{eq:defmaxfnct}Mf(x):=\sup_{I\subset \mathbb{R}}\frac{\chi_I(x)}{|I|}\int_I|f(s)|ds\Ee
where the supremum is taken over all intervals of $\mathbb{R}$.
% and for any measurable set $E\subset \mathbb{R}$, $\chi_E$ and $|E|$ denote respectively the indicator function of $E$ and the Lebesgue measure of $E$.
\vskip .1cm
Let us consider the following system of dyadic grids,
$$\mathcal D^\beta:=\{2^j\left([0,1)+m+(-1)^j\beta\right):m\in \mathbb Z,\,\,\,j\in \mathbb Z \},\,\,\,\textrm{for}\,\,\,\beta\in \{0,1/3\}.$$
%For more on this system of dyadic grids and its applications, we refer to \cite{AlPottReg, HyPerez,Lerner,LerOmbroPerezetal,PR}. 
When $\beta=0$, we observe that $\mathcal D^0$ is the standard dyadic grid of $\mathbb R$, denoted $\mathcal D$.
\vskip .2cm
For any $\beta\in \{0,1/3\}$, we denote by ${M}^{d,\beta}$ the dyadic analogue of the Hardy-Littlewood maximal function, defined as in (\ref{eq:defmaxfnct}) but with the supremum taken over dyadic intervals in the grid $\mathcal D^\beta$.
\vskip .2cm
It is a classical fact that for any locally integrable function $f$ on $\mathbb{R}$, \Be\label{eq:strongdyamax}Mf(x)\le 6\sum_{\beta\in \{0,\frac 13\}}{M}^{d,\beta}f(x).\Ee
The following is a well known result (see for example \cite{koki}). We provide a proof here for the sake of the reader.
\begin{prop}\label{prop:boundedHLmax}
Let $\Phi$ be a $\mathcal{C}^1$ convex growth function that satisfies the $\nabla_2$-condition. Then there exists a constant $C=C_\Phi>0$ such that for any $f\in L^\Phi(\mathbb{R})$,
\Be\label{eq:HLineq}
\int_{\mathbb{R}}\Phi(Mf(x))dx\le C\int_{\mathbb{R}}\Phi(|f(x)|)dx.
\Ee
\end{prop}
\begin{proof}
From the inequality (\ref{eq:strongdyamax}), it is enough to prove (\ref{eq:HLineq}) for the maximal function ${M}^{d,\beta}$, $\beta=0,\frac 13$. From standard properties of dyadic intervals, one obtain that 
 \Be\label{eq:weakmaxest} 	| \{ x\in {\mathbb{R}} :   M^{d,\beta}f(x) > \lambda  \} | \leq \frac{2}{\lambda} \int_{\{ t\in {\mathbb{R}} :  \mid f(t)  \mid > \frac{\lambda}{2}  \}}\mid f(t) \mid dt.\Ee
Hence
$$\begin{array}{rcl}
\int_{{\mathbb{R}}}\Phi\left( M^{d,\beta}f(x) \right) dx &=&  \int_{0}^{\infty}\Phi^{\prime}(\lambda) | \{ x\in {\mathbb{R}} :   M^{d,\beta}f(x) > \lambda  \} | d\lambda \\\\	
&\leq& \int_{0}^{\infty}\Phi^{\prime}(\lambda) \left(\frac{2}{\lambda} \int_{\{ x\in {\mathbb{R}} :  \mid f(x)  \mid > \frac{\lambda}{2}  \}}\mid f(x) \mid dx  \right)d\lambda \\\\	
&\leq& 2\int_{{\mathbb{R}}} \mid f(x) \mid \left(\int_{0}^{2\mid f(x) \mid} \frac{\Phi^{\prime}(\lambda)}{\lambda}d\lambda\right)	dx. \\
\end{array}$$
As $\Phi$ satisfies the $\nabla_2$-condition, we have from Lemma \ref{lem:dini} that there exists $ C>0$ such that $\forall~ t >0$ ,    \[   \int_{0}^{t}\frac{\Phi(\lambda)}{\lambda^{2}}d\lambda \leq C \frac{\Phi(t)}{t} .\]
It follows from an integration by parts that  \[  \int_{0}^{t}\frac{\Phi^{\prime}(\lambda)}{\lambda}d\lambda \leq \frac{\Phi(t)}{t} + \int_{0}^{t}\frac{\Phi(\lambda)}{\lambda^{2}}d\lambda \leq C_{1}\frac{\Phi(t)}{t}. \] 
Thus 
\[    2\int_{{\mathbb{R}}} \mid f(x) \mid \left(\int_{0}^{2\mid f(x) \mid} \frac{\Phi^{\prime}(\lambda)}{\lambda}d\lambda\right)	dx  \leq C \int_{\pmb{\mathbb{R}}} \Phi(\mid f(x) \mid)dx \]
and consequently, \[ \int_{{\mathbb{R}}}\Phi( M^{d,\beta}f(x) ) dx \leq C_{1} \int_{{\mathbb{R}}} \Phi(\mid f(x) \mid)dx.            \]
 
\end{proof}
The nontangential maximal function $f^*$ of a function $f$ defined on $\mathbb{C}_+$ is given by 
\Be\label{eq:deefnintanmaxfunct}
f^*(x):=\sup_{z\in \Gamma(x)}|f(z)|
\Ee 
where $\Gamma(x):=\{z=t+iy\in \mathbb{C}_+:\,|t-x|<y\}$.
\vskip .1cm
As for classical Hardy spaces of the upper-half plane, we have the following characterization of Hardy-Orlicz spaces.
\begin{thm}\label{thm:nontangequivdef}
Let $\Phi$ be a $\mathcal{C}^1$ convex growth function that satisfies the $\nabla_2$-condition. Then $f\in H^\Phi(\mathbb{C}_+)$ if and only if $f^*\in L^\Phi(\mathbb{R})$. Moreover,
$$\|f\|_{H^\Phi}^{lux}\approx \|f^*\|_{L^\Phi}^{lux}.$$
\end{thm} 
\begin{proof}
Let assume that $f \in {H}^{\Phi}({\mathbb{C_{+}}})$. Then as
$\Phi\in \mathcal{U}$ and satisfies $\nabla_2$-condition, we obtain as in the case of classical Hardy spaces (see \cite{Jan}) that there exists a unique function $g \in {L}^{\Phi}({\mathbb{R}})$ such that \[ \forall~ z=x+iy \in {\mathbb{C_{+}}},~~ f(z)= \int_{{\mathbb{R}}}P_{y}(t)g(x-t)dt  \] 
where $P_{y}(x)=\frac{1}{\pi}\frac{x}{x^2+y^2}$ is the Poisson kernel.
Moreover, $\|f\|_{{H}^{\Phi}}^{lux}=\| g\|_{{L}^{\Phi}}^{lux}$.
\vskip .1cm
From \cite[Theorem 4.2]{garnett}, we know that \[     f^{\star}(t_{0}) \leq C Mg(t_{0}),~ \forall~ t_{0}\in {\mathbb{R}}. \]
Hence as $\Phi\in \mathcal{U}$ and satisfies the $\nabla_2$-condition, and $g \in {L}^{\Phi}({\mathbb{R}})$, it follows from Proposition \ref{prop:boundedHLmax} that 	\[  \int_{{\mathbb{R}}}\Phi( Mg(x) ) dx  \leq C_{1} \int_{{\mathbb{R}}}\Phi( |g(x)| ) dx .\] Thus \[   \int_{{\mathbb{R}}}\Phi( f^{\star}(x) ) dx \leq C C_{1} \int_{{\mathbb{R}}}\Phi( |g(x)| ) dx. \]  One deduces that $f^{\star} \in {L}^{\Phi}{\mathbb{R}})$, and  
$\| f^{\star}\|_{{L}^{\Phi}}^{lux} \leq C_{2} \|f\|_{{H}^{\Phi}}^{lux}$ since $\|f\|_{{H}^{\Phi}}^{lux}=\| g\|_{{L}^{\Phi}}^{lux}$.
\vskip .2cm 
Now suppose that $f^{\star} \in {L}^{\Phi}({\mathbb{R}})$. Observe that
$\forall~ y > 0 ,~ \forall~ x \in {\mathbb{R}}$, $$|f(x+iy) | \leq f^{\star}(x),$$ since $x+iy \in \Gamma(x)$. Hence $\forall y > 0$ , 
\[      \int_{{\mathbb{R}}}\Phi( |f(x+iy)| ) dx	\leq \int_{{\mathbb{R}}}\Phi( f^{\star}(x) ) dx. \]
Thus \[     \sup_{y> 0}\int_{{\mathbb{R}}}\Phi( |f(x+iy)| ) dx	\leq \int_{{\mathbb{R}}}\Phi( f^{\star}(x) ) dx	\]
%As $f^{\star} \in \textbf{L}^{\Phi}({\mathbb{R}})$, we deduce that $f \in \textbf{H}^{\Phi}(\pmb{\mathbb{C_{+}}})$ and $
and consequently, $\|f\|_{{H}^{\Phi}}^{lux}\leq \| f^{\star}\|_{{L}^{\Phi}}^{lux}$.
\vskip .1cm
 We conclude that \[   \|f\|_{{H}^{\Phi}}^{lux}\approx \| f^{\star}\|_{{L}^{\Phi}}^{lux}.	\]	
 
\end{proof}
Let us finish this subsection by giving an example of elements in Hardy-Orlicz spaces.
\begin{lem}\label{lem:testfuncthardyo}
Let $\Phi$ be a convex growth function. Then for any $z=x+iy\in \mathbb{C}_+$, the function $$f_z(w):=\Phi^{-1}\left(\frac 1y\right)\frac{y^2}{(w-\bar{z})^2}$$
is in $H^\Phi(\mathbb{C}_+)$. Moreover,  $\|f\|_{H^\Phi}\le \pi.$
\end{lem}
\begin{proof}
It is clear that $f_{z}$ is analytic on ${\mathbb{C_{+}}}$. We observe that
 $$|\omega-\overline{z}|^{2}=(u-x)^{2}+(y+v)^{2} > y^{2} \Longrightarrow \dfrac{y^{2}}{|z-\overline{\omega}|^{2}}< 1.$$
As the function $t\rightarrow \frac{\Phi(t)}{t}$ is increasing, we obtain using Lemma \ref{lem:integkernel} that
$\forall~ v > 0$,  
$$\begin{array}{rcl}
\int_{{\mathbb{R}}}\Phi(|f_{z}(u+iv)|)du &=& \int_{{\mathbb{R}}}\Phi\left(\Phi^{-1}\left(\dfrac{1}{y}\right) \dfrac{y^{2}}{ |(u-x)+i(y+v) |^{2}}\right) du \\\\
&\leq& \int_{{\mathbb{R}}}\dfrac{y^{2}}{ |(u-x)+i(y+v) |^{2}}\Phi\left(\Phi^{-1}\left(\dfrac{1}{y}\right)\right)du, \\\\
&=& \int_{{\mathbb{R}}}\dfrac{y}{ |(u-x)+i(y+v) |^{2}}du \\\\
&=& y B(\frac{1}{2}, \frac{1}{2} )\frac{1}{y+v} \\\\
&\leq& \pi. 
\end{array}$$
%where $B(\cdot,\cdot)$ is the usual beta function. %$=B(n,m) =\int_{0}^{\infty}\dfrac{u^{m-1}}{(1+u)^{m+n}}du , ~~\forall~ m, n > 0$. 
Thus 	
\[    \sup_{v > 0}\int_{{\mathbb{R}}}\Phi(|f_{z}(u+iv)|)du \leq \pi < \infty. \]
That is $f_{z} \in {H}^{\Phi}({\mathbb{C_{+}}})$ and $\|f_{z}\|_{{H}^{\Phi}}^{lux} \leq \pi$.	

\end{proof}
\subsection{Some useful facts on Bergman-Orlicz spaces of the upper-half plane}
We start by observing that as in the case of Hardy-Orlicz spaces, the following holds.
\begin{lem}
Let $\Phi$ be a convex growth function, and let $\alpha>-1$. Then $\|f\|_{A_\alpha^\Phi}=0$ if and only if $f=0$.
\end{lem}

For any $\alpha>-1$, and any measurable set $E\subset \mathbb{C}_+$, we use the notation $$|E|_\alpha=V_\alpha(E)=\int_EdV_\alpha.$$ Let us prove the following pointwise estimate.
\begin{lem}\label{lem:pointwiseberg}
Let $\Phi$ be a convex growth function, and $\alpha > -1$. Then there exists $C=C_{\alpha}>0$  such that for any $ f\in {A}^{\Phi}_{\alpha}({\mathbb{C_{+}}})$ and any $z= x+iy \in {\mathbb{C_{+}}}$,  
\Be\label{eq:pointwiseBerg}|f(z)|~\leq~C \Phi^{-1}\left(\frac{1}{ y^{\alpha+2}}\right)\|f\|_{\Phi,\alpha}^{lux}.\Ee	
\end{lem}

\begin{proof}
Let $f \in {A}^{\Phi}_{\alpha}({\mathbb{C_{+}}})$. If $f=0$, then there is nothing to prove . Assume that $f\neq 0$. Let $z_{0}=x_{0}+iy_{0}\in {\mathbb{C_{+}}}$ and let
 $Q_{I}$ be the Carleson square centered at $z_{0}$. As $f$ is analytic, as a consequence of the mean value theorem, there exists a constant $C=C_\alpha>0$ and independent of $z_0$ such that 	
\Be\label{eq:mvt} |f(z_{0})|\le \dfrac{C}{|Q_{I}|_\alpha} \int_{Q_{I}}|f(u+iv)|dV_{\alpha}(u+iv)             \Ee    
(see \cite[Lemma 7.1]{BekSeh}).	
It follows from this, the Jensen's inequality and (\ref{uppertype}) that \[  \Phi\left(\dfrac{|f(z_{0})|}{\|f\|_{\Phi,\alpha}^{lux}}\right) \leq \dfrac{C}{|Q_{I}|_\alpha}\int_{{Q_I}}\Phi\left(\dfrac{|f(u+iv)|}{\|f\|_{\Phi,\alpha}^{lux}}\right)dV_{\alpha}(u+iv).  \]	
But \[   |Q_{I}|_\alpha=\int_{Q_{I}}dV_{\alpha}(u+iv)= \int_{0}^{|I|}\int_{I}v^{\alpha}du dv= \frac{1}{1+\alpha}|I|^{\alpha+2}=\frac{2^{\alpha+2}}{1+\alpha}y^{\alpha+2}_{0}. \]	
Hence \[  \Phi\left(\dfrac{|f(z_{0})|}{\|f\|_{\Phi,\alpha}^{lux}}\right) \leq \dfrac{C}{y^{\alpha+2}_{0}}         \]	
which leads to \[|f(z)|~\leq~C \Phi^{-1}\left(\frac{1}{ y^{\alpha+2}}\right)\|f\|_{\Phi,\alpha}^{lux},\,\,\,\textrm{for any}\,\,\,z=x+iy\in \mathbb{C}_+. \]
\end{proof}

Let $\alpha>-1$. We recall that the (weighted) Hardy-Littlewood maximal function of $\mathbb{C}_+$ is the function defined for any locally integrable function $f$ by
$$\mathcal{M}_\alpha f(x):=\sup_{I\subset \mathbb{R}}\frac{\chi_{Q_I}(x)}{|Q_I|_\alpha}\int_{Q_I}|f(w)|dV_\alpha(w)$$
where again, the supremum is taken over all intervals of $\mathbb{R}$. Its dyadic counterpart called dyadic (weighted) Hardy-Littlewood maximal function and denoted $\mathcal{M}_\alpha^d$ is  defined the same way but with supremum taken only over dyadic intervals of $\mathbb{R}.$
\vskip .1cm
Let us recall three useful facts, the first one is given in \cite[Lemma 2.2]{sehba1} (see also\cite[Lemma 3.4]{carnotbenoit}), the second one and the third one are pretty classical and can be found in \cite[Lemma 2.1]{sehba1}.
\begin{lem}\label{lem:levelsets}
Let $\alpha>-1$. Then for any locally integrable function $f$, the following assertions are satisfied.
\begin{itemize}
\item[(i)] There is a constant $C=C_\alpha>0$ such that for any $\lambda>0$,
$$\{z\in \mathbb{C}_+:\mathcal{M}_\alpha f(z)>\lambda\}\subset \{z\in \mathbb{C}_+:\mathcal{M}_\alpha^d f(z)>\frac{\lambda}{68}\}.$$
\item[(ii)] For any $\lambda>0$, there exists a family of disjoint maximal (with respect to inclusion) dyadic intervals $\{I_j\}_j$ such that 
$$\{z\in \mathbb{C}_+:\mathcal{M}_\alpha^d f(z)>\lambda\}=\bigcup_{j}Q_{I_j}.$$
\item[(iii)] There exists a constant $C=C_\alpha>0$ such that for any $\lambda>0$,
$$|\{z\in \mathbb{C}_+:\mathcal{M}_\alpha^d f(z)>\lambda\}|_\alpha\le \frac{C}{\lambda}\int_{\{z\in \mathbb{C}_+:|f(z)|>\frac {\lambda}2\}}|f(z)|dV_\alpha.$$
\end{itemize}
\end{lem}
Note that the dyadic intervals in assertions (ii) are maximal intervals such that $$\frac{1}{|Q_{I_j}|_\alpha}\int_{Q_{I_j}}|f(w)|dV_\alpha(w)>\lambda.$$
\vskip .1cm
Let us give a proof of the following result.
\begin{prop}\label{prop:boundedHLmaxberg}
Let $\Phi$ be a $\mathcal{C}^1$ convex growth function, and $\alpha>-1$. Assume that $\Phi$ satisfies the $\nabla_2$-condition. Then there exists a constant $C=C_\Phi>0$ such that for any $f\in L^\Phi(\mathbb{C}_+)$,
\Be\label{eq:HLineqberg}
\int_{\mathbb{C}_+}\Phi(\mathcal{M}_\alpha f(z))dV_\alpha(z)\le C\int_{\mathbb{C}_+}\Phi(|f(z)|)dV_\alpha(z).
\Ee
\end{prop}
\begin{proof}
Using assertions (i) and (iii) of the previous result and Lemma \ref{lem:dini}, we obtain
\Beas
L &:=& \int_{\mathbb{C}_+}\Phi(\mathcal{M}_\alpha f(z))dV_\alpha(z)\\ &=& \int_0^\infty\Phi'(\lambda)|\{z\in \mathbb{C}_+:\mathcal{M}_\alpha f(z)>\lambda\}|_\alpha d\lambda\\ &\le& \int_0^\infty\Phi'(\lambda)|\{z\in \mathbb{C}_+:\mathcal{M}_\alpha^df(z)>\frac{\lambda}{C}\}|_\alpha d\lambda\\ &\le& \int_0^\infty\Phi'(\lambda)\left(\frac{C}{\lambda}\int_{\{z\in \mathbb{C}_+:|f(z)|>\frac{\lambda}2\}}|f(z)|dV_\alpha(z)\right)d\lambda\\ &=& C\int_{\mathbb{C}_+}|f(z)|\left(\int_0^{2|f(z)|}\frac{\Phi'(\lambda)}{\lambda}\right)dV_\alpha(z)\\ &\approx& C\int_{\mathbb{C}_+}|f(z)|\left(\int_0^{2|f(z)|}\frac{\Phi(\lambda)}{\lambda^2}\right)dV_\alpha(z)\\ &\le& C\int_{\mathbb{C}_+}\Phi(|f(z)|)dV_\alpha(z).
\Eeas
\end{proof}
Let us observe that for $f$ locally integrable, $$\mathcal{M}_\alpha^df(z)\le \mathcal{M}_\alpha f(z),\,\,\,\textrm{for any}\,\,\,z\in \mathbb{C}_+$$
and that by (\ref{eq:mvt}) there exists a constant $C=C_\alpha>0$ such that
$$ |f(z)|\le C\mathcal{M}_\alpha f(z),\,\,\,\textrm{for any}\,\,\,z\in \mathbb{C}_+.$$
Combining these two facts with assertion (i) of Lemma \ref{lem:levelsets} and Proposition \ref{prop:boundedHLmaxberg}, we obtain the following.
\begin{cor}\label{cor:equivdefbergorlicz}
Let $\Phi$ be a $\mathcal{C}^1$ convex growth function, and $\alpha>-1$. Assume that $\Phi$ satisfies the $\nabla_2$-condition. Then for any holomorphic function $f$ on $\mathbb{C}_+$, the following are equivalent.
\begin{itemize}
\item[(i)] $f\in L^\Phi(\mathbb{C}_+,dV_\alpha)$.
\item[(ii)] $\mathcal{M}_\alpha f\in L^\Phi(\mathbb{C}_+,dV_\alpha)$.
\item[(iii)] $\mathcal{M}_\alpha^d f\in L^\Phi(\mathbb{C}_+,dV_\alpha)$.
\end{itemize}
\end{cor}
Obviously, the corresponding norms in the above corollary are equivalent and this provides equivalent definitions of Bergman-Orlicz spaces in terms of Hardy-Littlewood maximal functions.
\vskip .2cm
The following provides an example of function in the Bergman-Orlicz spaces.
\begin{lem}\label{lem:testfunctbergo}
Let $\Phi$ be a convex growth function, and $\alpha>-1$. Then for any $z=x+iy\in \mathbb{C}_+$, the function $$f(w):=\Phi^{-1}\left(\frac 1{y^{2+\alpha}}\right)\frac{y^{4+2\alpha}}{(w-\bar{z})^{4+2\alpha}}$$
belongs to $A_\alpha^\Phi(\mathbb{C}_+)$. Moreover, $\|f\|_{A^\Phi}\le B(\frac{1}{2}, \frac{3+2\alpha}{2})B(1+\alpha, 2+\alpha).$
\end{lem}
\begin{proof}
First observing that $\frac{y^{4+2\alpha}}{(w-\bar{z})^{4+2\alpha}}\le 1$ and using Lemma \ref{lem:integkernel}, we obtain
$$\begin{array}{rcl}
\int_{{\mathbb{C}}_{+}}\Phi(|f_{z}(\omega)|)dV_{\alpha}(\omega)
&=&\int_{{\mathbb{C}}_{+}}\Phi\left(\Phi^{-1}\left(\frac{1}{y^{2+\alpha}}\right) \frac{y^{4+2\alpha}}{ |\omega-\overline{z}|^{4+2\alpha}}\right) dV_{\alpha}(\omega) \\\\
&\leq& \int_{{\mathbb{C}}_{+}} \frac{y^{4+2\alpha}}{ |\omega-\overline{z}|^{4+2\alpha}}\Phi\left(\Phi^{-1}\left(\frac{1}{y^{2+\alpha}}\right)\right) dV_{\alpha}(\omega)\\\\ 
&\leq& \int_{0}^{\infty}\int_{{\mathbb{R}}}\frac{y^{2+\alpha}}{ |(u-x)+i(y+v) |^{4+2\alpha}}v^{\alpha}du dv \\\\
&=& y^{2+\alpha}\int_{0}^{\infty}\left( \int_{{\mathbb{R}}}\frac{du}{ |(u-x)+i(y+v) |^{4+2\alpha}}\right) v^{\alpha}dv \\\\
&\leq&  y^{2+\alpha}\int_{0}^{\infty}B(\frac{1}{2}, \frac{3+2\alpha}{2})\frac{1}{(y+v)^{3+2\alpha}}v^{\alpha}dv. \end{array}$$
Hence using Lemma \ref{lem:betafunctionconditions}, we obtain
$$\begin{array}{rcl}
\int_{{\mathbb{C}}_{+}}\Phi(|f_{z}(\omega)|)dV_{\alpha}(\omega)
&\le& B(\frac{1}{2}, \frac{3+2\alpha}{2})\frac{1}{y}\int_{0}^{\infty}\frac{(\frac{v}{y})^{\alpha}}{(1+\frac{v}{y})^{3+2\alpha}}dv	\\\\
&=& B(\frac{1}{2}, \frac{3+2\alpha}{2})\int_{0}^{\infty}\frac{u^{\alpha}}{(1+u)^{3+2\alpha}}du \\\\ 
&=& B(\frac{1}{2}, \frac{3+2\alpha}{2})B(1+\alpha, 2+\alpha).
\end{array}$$
Thus \[        \int_{{\mathbb{C}}_{+}}\Phi(|f_{z}(\omega)|)dV_{\alpha}(\omega) \leq B(\frac{1}{2}, \frac{3+2\alpha}{2})B(1+\alpha, 2+\alpha).	\]
Hence $f_{z}$ is uniformly in ${A}^{\Phi}_{\alpha}({\mathbb{C_{+}}})$ with \[   \|f_{z}\|_{{A}_\alpha^{\Phi}} \leq B(\frac{1}{2}, \frac{3+2\alpha}{2})B(1+\alpha, 2+\alpha). \] 

\end{proof}
\section{Proof of Carleson embeddings}
\subsection{A general characterization}
Let $s>0$. We prove here a characterization of the positive measures $\mu$ on $\mathbb{C}_+$ for which there is a constant $C>0$ such that for any finite interval $I\subset \mathbb{R}$,
\Be\label{eq:sphicarlmeasdef}
\mu(Q_I)\le \frac{C}{\Phi\left(\frac{1}{|I|^s}\right)}.
\Ee
If a measure $\mu$ satisfies (\ref{eq:sphicarlmeasdef}), we call $\mu$ a $s$-$\Phi$-Carleson measure. When $s=1$ this corresponds to $\Phi$-Carleson measures and for $s=2+\alpha$ with $\alpha>-1$, we recover the $(\Phi,\alpha)$-Carleson measures. When $\Phi(t)=t$, the above measures are usually called $s$-Carleson measures.
\vskip .1cm
We have the following equivalent definition of $s$-$\Phi$-Carleson measures.
\begin{thm}\label{thm:sphi}
Let $\Phi_1,\Phi_2$ be two convex growth functions with $\Phi_2\in \mathscr{U}$. Let $s>0$. Let $\mu$ be a positive Borel measure on $\mathbb{C}_+$. Then the following assertions are equivalent.
\begin{itemize}
\item[(a)] $\mu$ is a $s$-$\Phi_2\circ\Phi_1^{-1}$-Carleson measure. 
\item[(b)] There exists a constant $C>0$ such that 
\Be\label{eq:sphi}
\sup_{z=x+iy\in \mathbb{C}_+}\int_{\mathbb{C}_+}\Phi_2\left(\Phi_1^{-1}\left(\frac{1}{y^s}\right)\frac{y^{2s}}{|z-\bar{w}|^{2s}}\right)d\mu(w)\le C<\infty.
\Ee
\end{itemize}
Moreover, the constants in (\ref{eq:sphicarlmeasdef}) and (\ref{eq:sphi}) are equivalent.
\end{thm}
\begin{proof}
$(b)	\Rightarrow (a)$: 
Let $I\subset {\mathbb{R}}$ be a finite interval and  $Q_{I}$ its associated Carleson square. Assume that $Q_{I}$ is centered at $z_{0}=x_{0}+iy_{0}	\in {\mathbb{C}}_{+}$. Observe that for any $\omega \in Q_{I}$ , \[ \dfrac{1}{10^{s}} \leq \dfrac{y^{2s}_{0}}{ |\omega-\overline{z_{0}}|^{2s}} \leq 1.        \]	
As $|I|=2y_{0}$ and  $\Phi_{1}^{-1}$ is nondecreasing, it follows that \[ \Phi_{1}^{-1}\left(\frac{1}{|I|^{s}}\right)  = \Phi_{1}^{-1}\left(\frac{1}{2^{s}y_{0}^{s}}\right)  \leq  \Phi_{1}^{-1}\left(\frac{1}{y_{0}^{s}}\right).        \]
Hence \[  \dfrac{1}{10^{s}} \Phi_{1}^{-1}\left(\frac{1}{|I|^{s}}\right) \leq \Phi_{1}^{-1}\left(\frac{1}{y_{0}^{s}}\right)\dfrac{y^{2s}_{0}}{ |\omega-\overline{z_{0}}|^{2s}}.  \]	
As $\Phi_2\in \mathscr{U}$, using (\ref{uppertype}), we obtain
$$\begin{array}{rcl}
\Phi_{2} \circ \Phi_{1}^{-1}\left(\frac{1}{|I|^{s}}\right)\mu(Q_{I}) &=& \int_{Q_{I}}\Phi_{2} \circ \Phi_{1}^{-1}\left(\frac{1}{|I|^{s}}\right)d\mu(\omega) \\\\
&\leq& C\int_{Q_{I}}\Phi_{2}\left(\Phi^{-1}_{1}\left(\dfrac{1}{y^{s}_{0}}\right) \dfrac{y^{2s}_{0}}{ |\omega-\overline{z_{0}}|^{2s}}\right)d\mu(\omega) \\\\	
&\leq& C \int_{\pmb{\mathbb{C_{+}}}}\Phi_{2}\left(\Phi^{-1}_{1}\left(\dfrac{1}{y^{s}_{0}}\right) \dfrac{y^{2s}_{0}}{ |\omega-\overline{z_{0}}|^{2s}}\right)d\mu(\omega) \\\\	
&\le& C.  \\\\	
\end{array}	$$
We conclude that there is a constant $C>0$ such that for any interval $I\subset \mathbb{R}$, \[ \mu(Q_{I}) \leq \dfrac{ C}{\Phi_{2} \circ \Phi_{1}^{-1}(\frac{1}{|I|^{s}})}.             \]
That is $\mu$ is $s-\Phi_{2} \circ \Phi_{1}^{-1}$- Carleson measure.
\vskip .3cm
We next prove the reverse implication.
\vskip .1cm
$(a) \Rightarrow (b)$:
Assume that $\mu$ is a  $s-\Phi_{2} \circ \Phi_{1}^{-1}$-Carleson. Let
 $z_{0}=x_{0}+iy_{0} \in {\mathbb{C}}_{+}$ be fixed, and define $I_0$ to be the interval about $x_0$ and length $2y_0$. For any $j\in \mathbb{N}$, define $I_{j}\subset {\mathbb{R}}$ to be the interval centered at $x_{0}$ with length $2^j|I_{0}|$. Let $Q_{I_{j}}$ be the Carleson square associated to $I_{j}$. For $j=1,2,\ldots$, put
\[  E_{j}=Q_{I_{j}}\backslash Q_{I_{j-1}}~~~and~~~ E_{0}=Q_{I_{0}} \]
Then for $j \geq 0$ and $\omega \in E_{j}$,  
\[  \dfrac{y^{2}_{0}}{ |\omega-\overline{z_{0}}|^{2}} \leq \dfrac{1}{2^{2(j-1)}}\]
and $\mu(E_j)\le \mu(Q_{j}).$
\vskip .2cm
%For $\Phi\in \mathcal{L}\cup\mathcal{U}$, define
%\[\varepsilon_{\Phi}=  \left\{
%\begin{array}{ll} 1 & \mbox{ if~ $\Phi\in \mathcal{U}$}\\ p  & \mbox{ if~ $ \Phi\in \mathcal{L}_{p}~~ with~~ 0 < p<1.   $}
%\end{array}
%\right.             \]
Using (\ref{uppertype}), we obtain 
$$\begin{array}{rcl}
T &:=& \int_{{\mathbb{C}}_{+}}\Phi_{2}\left(\Phi^{-1}_{1}\left(\dfrac{1}{y^{s}_{0}}\right) \dfrac{y^{2s}_{0}}{ |\omega-\overline{z_{0}}|^{2s}}\right) d\mu(\omega)\\ &=& \sum_{j=0}^{\infty}	\int_{E_{j}}\Phi_{2}\left(\Phi^{-1}_{1}\left(\dfrac{1}{y^{s}_{0}}\right) \dfrac{y^{2s}_{0}}{ |\omega-\overline{z_{0}}|^{2s}}\right) d\mu(\omega) \\\\
&\le& \sum_{j=0}^{\infty}	\int_{E_{j}}\Phi_{2}\left(\Phi^{-1}_{1}\left(\dfrac{1}{y^{s}_{0}}\right)\dfrac{1}{2^{2s(j-1)}} \right) d\mu(\omega) \\\\
&=& \sum_{j=0}^{\infty}	\int_{E_{j}}\Phi_{2}\left(\Phi^{-1}_{1}\left(\dfrac{1}{y^{s}_{0}}\right)\dfrac{4^s}{2^{sj}} \dfrac{1}{2^{s(j+1)}}\right) d\mu(\omega)\\\\ &\le& C\sum_{j=0}^{\infty}2^{-sj}	\int_{E_{j}}\Phi_{2}\left(\Phi^{-1}_{1}\left(\dfrac{1}{2^{s(j+1)}y^{s}_{0}}\right) \right) d\mu(\omega)\\\\
&\le& C\sum_{j=0}^{\infty}2^{-sj}\Phi_{2}\circ\Phi^{-1}_{1}\left(\dfrac{1}{ |I_j|^s}\right) \mu(Q_{I_j}) \\\\	
&\leq&  C\sum_{j=0}^{\infty}2^{-sj} \\\\	
&\leq&  C	
\end{array}$$
and the last constant does not depend on $y_0$. We conclude that
 \[ \sup_{z=x+iy \in {\mathbb{C_{+}}}} \int_{{\mathbb{C}}_{+}}\Phi_{2}\left(\Phi^{-1}_{1}\left(\dfrac{1}{y^{s}}\right) \dfrac{y^{2s}}{ |\omega-\overline{z}|^{2s}}\right) d\mu(\omega) \leq \tilde{C} < \infty     . \]
 The proof is complete.
\end{proof}
\subsection{Proof of Theorem \ref{thm:main1}}
For any measurable set $E\subset \mathbb{R}$, we denote by $|E|$ the Lebesgue measure of $E$. We start with the following crucial lemma.
\begin{lem}\label{lem:main11}
Let $\Phi$ be a growth function such that the function $t\mapsto \tilde{\Phi}(t):=\frac{1}{\Phi\left(\frac 1t\right)}$ belongs to the class $\mathscr{U}$. Assume that $\mu$ is a $\Phi$-Carleson measure. Then for any harmonic function $f$ on $\mathbb{C}_+$ and any $\lambda>0$,
\Be\label{eq:main11}
\mu\left(\{z\in \mathbb{C}_+: |f(z)|>\lambda\}\right)\le C\tilde{\Phi}\left(|\{x\in \mathbb{R}: f^*(x)>\lambda\}|\right)
\Ee
where $C$ is the constant in (\ref{eq:phicarldef}). Moreover, if $\Phi\in \mathscr{U}$ and satisfies the $\nabla_2$-condition, then the reverse holds. That is if $\mu$ satisfies (\ref{eq:main11}), then $\mu$ is a $\Phi$-Carleson measure with the same constant.
\end{lem}
\begin{proof}
Assume that $\mu$ is a $\Phi$-Carleson measure. Fix $\lambda > 0$. 
%$(I_{j})_{j}$ une famille d'intervalle dyadique maximal telle que \[  \{ x\in \pmb{\mathbb{R}} : f^{\star}(x) >  \lambda \}= \cup_{j}I_{j} \]	
We start by observing that the set $$E_\lambda:=\{ t\in {\mathbb{R}} : f^{\star}(t) >  \lambda \}$$
is open and consequently, is a disjoint union of open intervals $\{I_j\}$ (see \cite[Page 138]{grafakos}). 
\vskip .2cm
If $z=x+iy\in E_\lambda$, then  $f^{\star}(t) >  \lambda$ for any $t$ in the interval $I_{z}:=\{t\in \mathbb{R}:\,|t-x|<y\}$. Hence there is a unique $j_0$ such that the interval $I_z$ is contained in $I_{j_0}$. Moreover, if  $Q_{I_{j_{0}}}$ is the Carleson square associated to $I_{j_{0}}$, then $z \in Q_{I_{j_{0}}}$. Thus
 \[  \{ z\in {\mathbb{C_{+}}} : |f(z)| >  \lambda \} \subset \bigcup_{j}Q_{I_{j}} .  \]
It follows that
$$\mu(\{ z\in {\mathbb{C_{+}}} : |f(z)| >  \lambda \} ) \leq \sum_{j}\mu(Q_{I_{j}})\leq \sum_{j} \dfrac{C}{\Phi(\frac{1}{|I_{j}|})}\leq C \sum_{j} \tilde{\Phi}(|I_{j}|).$$
As $\tilde{\Phi} \in \mathcal{U}$, we have 
\[    \sum_{j} \tilde{\Phi}(|I_{j}|) \leq \tilde{\Phi}( \sum_{j} |I_{j}|) = \tilde{\Phi}(| \bigcup_{j}I_{j} | ) = \tilde{\Phi}(| \{ x\in {\mathbb{R}} : f^{\star}(x) >  \lambda \} |  ). \] Hence \[ \mu(\{ z\in {\mathbb{C_{+}}} : |f(z)| >  \lambda \} )  \leq C \tilde{\Phi}(| \{ x\in {\mathbb{R}} : f^{\star}(x) >  \lambda \} |).   \]
\vskip .3cm
Let us now assume that $\Phi\in \mathscr{U}$ and satisfies the $\nabla_2$-condition and that (\ref{eq:main11}) holds. Let
$I\subset {\mathbb{R}}$ be an interval and $Q_{I}$ its associated Carleson square.
For $\lambda > 0$ given, define    $f=4 \lambda \chi_{I}$. Then $f \in {L}^{\Phi}({\mathbb{R}})$.   Consider the function \[ u(z)=P_{y}\star f(x) = \int_{{\mathbb{R}}}P_{y}(x-t)f(t)dt ,~ \forall~ z=x+iy \in {\mathbb{C_{+}}}.  \]
Then $\forall~ z \in Q_{I} ,~ u(z) >  \lambda$. Hence 
\[     Q_{I} \subset \{ z\in {\mathbb{C_{+}}} : |u(z)| >  \lambda \}.\] 
Using Proposition \ref{prop:boundedHLmax}, we obtain 
\Beas
\mu( Q_{I} ) &\leq& \mu(\{ z\in {\mathbb{C_{+}}} : |u(z)| >  \lambda \}) 	
\leq C\tilde{\Phi}(| \{ x\in {\mathbb{R}} : u^{\star}(x) >  \lambda \} |) \\\\
&=& C\tilde{\Phi}\left(| \{ x\in {\mathbb{R}} :\Phi(u^{\star}(x)) >  \Phi(\lambda )\} |\right)\\ 
&\leq& C\tilde{\Phi}\left( \frac{1}{\Phi (\lambda)} \int_{{\mathbb{R}}}\Phi(u^{\star}(x))dx \right)\\\\
&\leq& C\tilde{\Phi}\left( \frac{1}{\Phi (\lambda)} \int_{{\mathbb{R}}}\Phi(Mf(x))dx \right)
\leq C\tilde{\Phi}\left( \frac{1}{\Phi (\lambda)} \int_{{\mathbb{R}}}\Phi(f(x))dx \right)\\
&\leq& C\tilde{\Phi}(| I | ) 
= \dfrac{C}{\Phi(\frac{1}{| I |})}.
\Eeas
Thus $\mu$ is a $\Phi$-Carleson measure. The proof is complete.
%\end{proof}	

\end{proof}
Let us now prove the Carleson embedding for Hardy-Orlicz spaces.
\begin{proof}[Proof of Theorem \ref{thm:main1}]
We have from  Theorem \ref{thm:sphi} that $(a)\Leftrightarrow (b)$. Hence it is enough to prove that $(a)	\Rightarrow (c) \Rightarrow (b) $. We start with the second implication.
\vskip .2cm
$(c) \Rightarrow (b)$:
$\forall z_{0}=x_{0}+iy_{0}\in {\mathbb{C}}_{+}$ , we have from Lemma \ref{lem:testfuncthardyo} that the function 	
\[ f_{z_{0}}(\omega)=\Phi^{-1}_{1}\left(\dfrac{1}{y_{0}}\right) \dfrac{y^{2}_{0}}{ (\omega-\overline{z_{0}})^{2}} ,~ \forall~ \omega=u+iv \in {\mathbb{C_{+}}}  \] 	
belongs to ${H}^{\Phi_{1}}({\mathbb{C_{+}}})$, and $\|f_{z_{0}}\|_{{H}^{\Phi_{1}}}^{lux} \leq \pi$.
It follows from assertion $(c)$ that there is a constant $K>0$ such that \[ \int_{{\mathbb{C}}_{+}}\Phi_{2}\left( \dfrac{|f_{z_{0}}(z)|}{K\|f_{z_{0}}\|_{{H}^{\Phi_{1}}}^{lux}}\right)d\mu(z) < \infty.             \]
This implies that there is  $C>0$ independent of $z_0$ such that \[ \int_{{\mathbb{C}}_{+}}\Phi_{2}( |f_{z_{0}}(z)|)d\mu(z)) \leq C < \infty.  \]	
We can then conclude that \[ \sup_{z=x+iy \in {\mathbb{C_{+}}}} \int_{{\mathbb{C}}_{+}}\Phi_{2}\left(\Phi^{-1}_{1}\left(\dfrac{1}{y}\right) \dfrac{y^{2}}{ |\omega-\overline{z}|^{2}}\right) d\mu(\omega) \leq C < \infty.      \]
\vskip .2cm
$(a)	\Rightarrow (c) $:
As $\Phi_{1}, \Phi_{2}\in \mathcal{U}$ and $\dfrac{\Phi_{2}}{\Phi_{1}}$ is nondecreasing, we have from Lemma \ref{lem:reverseprodphi} that the function : \[ \Phi_{3}(t) =\dfrac{1}{\Phi_{2} \circ \Phi_{1}^{-1}(\frac{1}{t})}   ,  ~ \forall~ t > 0     \] 
also belongs to  $\mathcal{U}$.
\vskip .1cm
Let $f \in {H}^{\Phi_{1}}({\mathbb{C_{+}}})$, $f\not=0$. As $\Phi_{1}\in \mathcal{U}$ and satisfies the $\nabla_2$-condition, we have by Theorem \ref{thm:nontangequivdef} that $f^{\star} \in {L}^{\Phi_{1}}({\mathbb{R}})$, and \[   \|f\|_{{H}^{\Phi_{1}}}^{lux}\approx \|f^{\star}\|_{{L}^{\Phi_{1}}}^{lux}. \]
Hence there is a constant $C> 1$ such that  $\|f^{\star}\|_{{L}^{\Phi_{1}}}^{lux} \leq C \|f\|_{{H}^{\Phi_{1}}}^{lux}$. It follows that
 
$$\begin{array}{rcl}
\int_{{\mathbb{C}}_{+}}\Phi_{2}\left( \dfrac{|f(z)|}{C\|f\|_{{H}^{\Phi_{1}}}^{lux}}\right)d\mu(z) 
&\leq& \int_{{\mathbb{C}}_{+}}\Phi_{2}\left( \dfrac{|f(z)|}{\|f^{\star}\|_{{L}^{\Phi_{1}}}^{lux}}\right)d\mu(z) \\\\	
&=& \int_{0}^{\infty}\Phi_{2}^{\prime}(\lambda)\mu(\{ z\in {\mathbb{C_{+}}} : |f(z)| >  \lambda \|f^{\star}\|_{{L}^{\Phi_{1}}}^{lux}\})d\lambda .	
\end{array}$$	
As $\mu$ is a $\Phi_{2} \circ \Phi_{1}^{-1}$-Carleson measure and $\Phi_{3}\in \mathcal{U}$, we have by Lemma \ref{lem:main11} that there is constant $K>0$ such that 
\[ \mu(\{ z\in {\mathbb{C_{+}}} : |f(z)| >  \lambda \|f^{\star}\|_{{L}^{\Phi_{1}}}^{lux}\} )  \leq K \Phi_{3}(| \{ x\in {\mathbb{R}} : f^{\star}(x) >  \lambda \|f^{\star}\|_{{L}^{\Phi_{1}}}^{lux}\} |).   \]
Let us put \[ E_{\lambda}= \{ x\in {\mathbb{R}} : f^{\star}(x) >  \lambda \|f^{\star}\|_{{L}^{\Phi_{1}}}^{lux}\}. \]
Then 
$$\begin{array}{rcl}
|E_{\lambda}|
&=& | \{ x\in {\mathbb{R}} : \Phi_{1}\left(\frac{f^{\star}(x)}{\|f^{\star}\|_{{L}^{\Phi_{1}}}^{lux}}   \right ) >  \Phi_{1}(\lambda) \} | \\\\	
&\leq& \frac{1}{\Phi_{1}(\lambda)}\int_{{\mathbb{R}}}\Phi_{1}\left(\frac{f^{\star}(x)}{\|f^{\star}\|_{{L}^{\Phi_{1}}}^{lux}}\right)dx \leq \dfrac{1}{\Phi_{1}(\lambda)}. 	
\end{array}$$	
As the function $t\mapsto \dfrac{\Phi_{3}(t)}{t}$ is nondecreasing, we deduce that \[ \Phi_{3}(|E_{\lambda}|) \leq \Phi_{1}(\lambda)\Phi_{3}\left(\dfrac{1}{\Phi_{1}(\lambda)}\right)|E_{\lambda}|. \]	
Hence 	
$$\begin{array}{rcl}
\int_{{\mathbb{C}}_{+}}\Phi_{2}\left( \dfrac{|f(z)|}{C\|f\|_{{H}^{\Phi_{1}}}^{lux}}\right)d\mu(z) &\leq& \int_{0}^{\infty}\Phi_{2}^{\prime}(\lambda)\mu(E_\lambda)d\lambda \\\\	
&\leq& K \int_{0}^{\infty}\Phi_{2}^{\prime}(\lambda)\Phi_{3}(|E_{\lambda}|)d\lambda \\\\	
&\leq& K \int_{0}^{\infty}\Phi_{2}^{\prime}(\lambda)\Phi_{1}(\lambda)\Phi_{3}(\dfrac{1}{\Phi_{1}(\lambda)})|E_{\lambda}|d\lambda \\\\	
&=& K \int_{0}^{\infty}\Phi_{2}^{\prime}(\lambda)\Phi_{1}(\lambda)\frac{1}{\Phi_{2}(\lambda)}|E_{\lambda}|d\lambda \\\\	
&\approx& \int_{0}^{\infty}\Phi_{1}^{\prime}(\lambda)|E_{\lambda}|d\lambda \\\\	
&\approx& \int_{0}^{\infty}\Phi_{1}^{\prime}(\lambda)| \{ x\in {\mathbb{R}} : \frac{f^{\star}(x)}{\|f^{\star}\|_{{L}^{\Phi_{1}}}^{lux}} >  \lambda \} |d\lambda \\\\ 	
&=&  \int_{{\mathbb{R}}}\Phi_{1}\left(\frac{f^{\star}(x)}{\|f^{\star}\|_{{L}^{\Phi_{1}}}^{lux}}\right)dx \\\\	
&\leq& 1 .\\\\	
\end{array}$$	
The proof is complete.
\end{proof}
\subsection{Proof of Theorem \ref{thm:main2}}
Let us start with the following key result.
\begin{lem}\label{lem:main21}
Let $\alpha>-1$. Let $\Phi$ be growth function such the function $t\mapsto \tilde{\Phi}(t):=\frac{1}{\Phi\left(\frac 1t\right)}$ belongs to the class $\mathscr{U}$. Assume $\mu$ is a $(\Phi,\alpha)$-Carleson measure. Then for any locally integrable function $f$ on $\mathbb{C}_+$ and any $\lambda>0$,
\Be\label{eq:main21}
\mu\left(\{z\in \mathbb{C}_+: \mathcal{M}_\alpha^d f(z)>\lambda\}\right)\le C\tilde{\Phi}\left(|\{z\in \mathbb{C}_+: \mathcal{M}_\alpha^d f(z)>\lambda\}|_\alpha\right)
\Ee
where $C$ is the constant in (\ref{eq:phicarlbergdef}). Moreover, if $\Phi\in \mathscr{U}$ and satisfies the $\nabla_2$-condition, then the reverse holds. That is if $\mu$ satisfies (\ref{eq:main21}), then $\mu$ is a $(\Phi,\alpha)$-Carleson measure with a constant equivalent to the one in (\ref{eq:main21}).
\end{lem}
\begin{proof}
Recall with Lemma \ref{lem:levelsets} that $$E_\lambda:=\{z\in \mathbb{C}_+: \mathcal{M}_\alpha^d f(z)>\lambda\}=\bigcup_jQ_{I_j}$$
where $\{I_j\}_j$ is a family of pairwise disjoint dyadic intervals. It follows easily that 
\Beas
\mu(E_\lambda) &=& \sum_j\mu(Q_{I_j})\le C\sum_j\tilde{\Phi}(|I|^{2+\alpha})\\ &\le& C\tilde{\Phi}\left(\sum_j|I|^{2+\alpha}\right)=C\tilde{\Phi}(E_\lambda).
\Eeas
For the converse, let $I$ be any interval in $\mathbb{R}$ and for $\lambda>0$, put $f(z)=\lambda\chi_{Q_I}(z)$. Then using the first assertion in Lemma \ref{lem:levelsets} and Proposition \ref{prop:boundedHLmaxberg}, we obtain
\Beas
\mu(Q_I) &\le& \mu\left(\{z\in \mathbb{C}_+: \mathcal{M}_\alpha^d f(z)>\lambda\}\right)\\ &\le& C\tilde{\Phi}\left(|\{z\in \mathbb{C}_+: \mathcal{M}_\alpha^d f(z)>\lambda\}|_\alpha\right)\\ &\le& C\tilde{\Phi}\left(\frac{1}{\Phi(\lambda)}\int_{\mathbb{C}_+}\Phi\left(\mathcal{M}_\alpha^d f(z)\right)dV_\alpha(z)\right)\\ &\le& C\tilde{\Phi}\left(\frac{1}{\Phi(\lambda)}\int_{\mathbb{C}_+}\Phi\left(|f(z)|\right)dV_\alpha(z)\right)\\ &=& C\tilde{\Phi}\left(\frac{1}{\Phi(\lambda)}\int_{\mathbb{C}_+}\Phi(\lambda)\chi_{Q_I}(z)dV_\alpha(z)\right)\\ &=& C\tilde{\Phi}(|Q_I|_\alpha)=\frac{C}{\Phi\left(\frac{1}{|I|^{2+\alpha}}\right)}.
\Eeas
The proof is complete.
\end{proof}
Next, we prove Theorem \ref{thm:main2}.
\begin{proof}[Proof of Theorem \ref{thm:main2}]
We note that the equivalence (a)$\Leftrightarrow$(b) is a special case of Theorem \ref{thm:sphi}. That (c)$\Rightarrow$(b) follows by taking as $f$ in (\ref{eq:equivcarlberg3}), the test function given in Lemma \ref{lem:testfunctbergo}. To finish, it suffices to prove that (a)$\Rightarrow$(c).
\vskip .3cm
Let us assume that $\mu$ is a $(\Phi_2\circ\Phi_1^{-1},\alpha)$-Carleson measure. Let $C$ be the constant in (\ref{eq:HLineqberg}). We can assume that $C>1$. Put $$\frac 1{\Phi_3(t)}:=\Phi_2\circ\Phi_1^{-1}\left(\frac 1t\right).$$ For $\lambda>0$, define $$ E_\lambda :=\left\{z\in \mathbb{C}_+:\,\mathcal{M}_\alpha^d\left(\frac{f}{C\|f\|_{A_\alpha^{\Phi_1}}^{lux}}\right)(z)>\lambda\right\}.$$
Then using the first assertion in Lemma \ref{lem:levelsets}, and Lemma \ref{lem:main21},
we obtain
\Beas
L &:=& \int_{\mathbb{C}_+}\Phi_2\left(\frac{|f(z)|}{C\|f\|_{A_\alpha^{\Phi_1}}^{lux}}\right)d\mu(z)\\ &\le& K\int_{\mathbb{C}_+}\Phi_2\left(\mathcal{M}_\alpha\left(\frac{f}{C\|f\|_{A_\alpha^{\Phi_1}}^{lux}}\right)(z)\right)d\mu(z)\\ &=& \int_0^\infty\Phi_2'(\lambda)\mu\left(\{z\in \mathbb{C}_+:\,\mathcal{M}_\alpha\left(\frac{f}{C\|f\|_{A_\alpha^{\Phi_1}}^{lux}}\right)(z)>\lambda\}\right)d\lambda\\
&\le& \int_0^\infty\Phi_2'(\lambda)\mu\left(\{z\in \mathbb{C}_+:\,\mathcal{M}_\alpha^d\left(\frac{f}{C\|f\|_{A_\alpha^{\Phi_1}}^{lux}}\right)(z)>\frac{\lambda}{68}\}\right)d\lambda\\ &\le& K\int_0^\infty\Phi_2'(\lambda)\mu\left(E_\lambda\right)d\lambda\\ &\le& K\int_0^\infty\Phi_2'(\lambda)\Phi_3\left(|E_\lambda|_\alpha\right)d\lambda.
\Eeas
Now we recall that by Lemma \ref{lem:reverseprodphi}, $\Phi_3$ also belongs to the class $\mathscr{U}$ and so the function $t\mapsto \frac{\Phi_3(t)}{t}$ is increasing. We also observe using Proposition \ref{prop:boundedHLmaxberg} that
\Beas
|E_\lambda|_\alpha &=& |\{z\in \mathbb{C}_+:\,\mathcal{M}_\alpha^d\left(\frac{f}{C\|f\|_{A_\alpha^{\Phi_1}}^{lux}}\right)(z)>\lambda\}|_\alpha\\ &\le& \frac{1}{\Phi_1(\lambda)}\int_{\mathbb{C}_+}\Phi_1\left(\mathcal{M}_\alpha^d\left(\frac{f}{C\|f\|_{A_\alpha^{\Phi_1}}^{lux}}\right)(z)\right)dV_\alpha\\ &\le& \frac{C}{\Phi_1(\lambda)}\int_{\mathbb{C}_+}\Phi_1\left(\frac{|f(z)|}{C\|f\|_{A_\alpha^{\Phi_1}}^{lux}}\right)dV_\alpha\\ &\le& \frac{1}{\Phi_1(\lambda)}.
\Eeas
Thus
\Beas
\int_{\mathbb{C}_+}\Phi_2\left(\frac{|f(z)|}{C\|f\|_{A_\alpha^{\Phi_1}}^{lux}}\right)d\mu(z) &\le& K\int_0^\infty\Phi_2'(\lambda)\Phi_3\left(|E_\lambda|_\alpha\right)d\lambda\\ &=& K\int_0^\infty\Phi_2'(\lambda)\frac{\Phi_3\left(|E_\lambda|_\alpha\right)}{|E_\lambda|_\alpha}|E_\lambda|_\alpha d\lambda\\ &\le& K\int_0^\infty\Phi_2'(\lambda)\Phi_1(\lambda)\Phi_3\left(\frac{1}{\Phi_1(\lambda)}\right)|E_\lambda|_\alpha d\lambda\\ &=& K\int_0^\infty\Phi_2'(\lambda)\frac{\Phi_1(\lambda)}{\Phi_2(\lambda)}|E_\lambda|_\alpha d\lambda\\ &\le& K\int_0^\infty\Phi_1'(\lambda)|E_\lambda|_\alpha d\lambda\\ &=& K\int_{\mathbb{C}_+}\Phi_1\left(\mathcal{M}_\alpha^d\left(\frac{f}{C\|f\|_{A_\alpha^{\Phi_1}}^{lux}}\right)(z)\right)dV_\alpha\\ &\le& CK\int_{\mathbb{C}_+}\Phi_1\left(\frac{|f(z)|}{C\|f\|_{A_\alpha^{\Phi_1}}^{lux}}\right)dV_\alpha\\ &\le& K.
\Eeas
The proof is complete.
\end{proof}
\section{{Embedding of Hardy-Orlicz spaces and Bergman-Orlicz spaces into Bergman-Orlicz spaces}}
In this part, we are interested in the conditions under which a Hardy-Orlicz space or Bergman-Orlicz space embeds continuously into another Bergman-Orlicz space.
\begin{proof}[Proof of Theorem \ref{thm:embed1}]
We start by recalling that if $I\subset {\mathbb{R}}$ is an interval and $Q_{I}$ its associated Carleson square, then	
\[      V_{\alpha}(Q_{I})= \frac{1}{1+\alpha}|I|^{\alpha+2}.       \]	
Now assume that ${H}^{\Phi_{1}}({\mathbb{C_{+}}})$ embeds continuously into  ${A}^{\Phi_{2}}_{\alpha}({\mathbb{C_{+}}})$. That is
there is a constant $C>0$ such that  for any $ f \in {H}^{\Phi_{1}}({\mathbb{C_{+}}}), f\not=0$,  \[ \int_{{\mathbb{C}}_{+}}\Phi_{2}\left( \dfrac{|f(z)|}{C\|f\|_{{H}^{\Phi_{1}}}^{lux}}\right)dV_{\alpha}(z) \leq 1.             \]	
As $\Phi_{1}, \Phi_{2}\in \mathcal{U}$, $\Phi_{1}$ satisfies the $\nabla_2$-condition and $\dfrac{\Phi_{2}}{\Phi_{1}}$ is nondecreasing, by Theorem \ref{thm:main1}, $V_{\alpha}$ is a $\Phi_{2} \circ \Phi_{1}^{-1}$-Carleson measure. 
\vskip .1cm	
For $t > 0$, let $I\subset {\mathbb{R}}$ be an interval such that $|I|=\dfrac{1}{t}$ and let	 $Q_{I}$ be the Carleson square associated to $I$. Then as $V_\alpha$ is a $\Phi_{2} \circ \Phi_{1}^{-1}$-Carleson measure, we obtain in particular that for some $C_1$ independent of $I$, 
 
$$\begin{array}{rcl}
V_{\alpha}(Q_{I}) \leq \dfrac{C_{1}}{\Phi_{2} \circ \Phi_{1}^{-1}(\frac{1}{|I|})} 
&\Leftrightarrow&   \frac{1}{1+\alpha}|I|^{\alpha+2} \leq \dfrac{C_{1}}{\Phi_{2} \circ \Phi_{1}^{-1}(\frac{1}{|I|})} \\\\	
&\Rightarrow&  \Phi_{2} \circ \Phi_{1}^{-1}(\frac{1}{|I|}) \leq \tilde{C}	\dfrac{1}{|I|^{\alpha+2}}. \\\\
\end{array}$$	
That is $$\Phi_{2} \circ \Phi_{1}^{-1}(t)	\leq \tilde{C}t^{\alpha+2}$$ or equivalently, $$\Phi_{1}^{-1}(t) \leq \Phi_{2}^{-1}( \tilde{C}t^{\alpha+2}).$$
\vskip .2cm
Conversely,
assume that there exists a constant $ C>0$ such that for any $ t > 0$ ,  $$\Phi_{1}^{-1}(t) \leq \Phi_{2}^{-1}(Ct^{2+\alpha}). $$ 
Let $I\subset {\mathbb{R}}$ be an interval and $Q_{I}$ its associted Carleson square. Then

$$\begin{array}{rcl}
\Phi_{1}^{-1}\left( \frac{1}{|I|} \right) \leq \Phi_{2}^{-1}\left(C\frac{1}{|I|^{\alpha+2}} \right)  	&\Leftrightarrow&  \Phi_{2} \circ \Phi_{1}^{-1}\left(\frac{1}{|I|}\right) \leq C\dfrac{1}{|I|^{\alpha+2}} \\\\
&\Leftrightarrow&  \Phi_{2} \circ \Phi_{1}^{-1}\left(\frac{1}{|I|}\right) \leq C\dfrac{1}{(\alpha+1)V_{\alpha}(Q_{I})} \\\\	
&\Leftrightarrow& V_{\alpha}(Q_{I}) \leq \dfrac{C_{1}}{\Phi_{2} \circ \Phi_{1}^{-1}\left(\frac{1}{|I|}\right)}. \\	
\end{array}$$	
That is $V_{\alpha}$ is a $\Phi_{2} \circ \Phi_{1}^{-1}$-Carleson measure. Thus by Theorem \ref{thm:main1}, there exists a constant $K>0$ such for any  $\ f \in {H}^{\Phi_{1}}({\mathbb{C_{+}}}),~ f\not=0$,  \[ \int_{{\mathbb{C}}_{+}}\Phi_{2}\left( \dfrac{|f(z)|}{K\|f\|_{{H}^{\Phi_{1}}}^{lux}}\right)dV_{\alpha}(z) < \infty,          \]
that is  ${H}^{\Phi_{1}}({\mathbb{C_{+}}})$ embeds continuously into  ${A}^{\Phi_{2}}_{\alpha}({\mathbb{C_{+}}})$. The proof is complete.

\end{proof}
\begin{proof}[Proof of Theorem \ref{thm:embed2}]
This essentially follows as above. We leave it to the interested reader.
\end{proof}
\section{Pointwise multipliers characterizations}
We start with the following lemma.
\begin{lem}\label{lem:main31}
Let $\Phi_1,\Phi_2\in {\mathscr U}$. Assume that  $\frac{\Phi_2}{\Phi_1}$ is non-decreasing. Let $\alpha>-1$ and define for $t\in (0,\infty)$, the function
$$\omega(t)=\frac{\Phi_2^{-1}\left(\frac{1}{t^{2+\alpha}}\right)}{\Phi_1^{-1}\left(\frac{1}{t}\right)}.$$
Then the following assertions hold.
\begin{itemize}
\item[(i)] If $\Phi_1$ satisfies the $\nabla_2$-condition, and $\omega$ is equivalent to $1$, then $$\mathcal{M}\left(H^{\Phi_1}(\mathbb{C}_+),A_\alpha^{\Phi_2}(\mathbb{C}_+)\right)=H^\infty(\mathbb{C}_+).$$
\item[(ii)] If $\omega$ is nondecreasing $\lim_{t\rightarrow 0}\omega(t)=0$, then $$\mathcal{M}\left(H^{\Phi_1}(\mathbb{C}_+),A_\alpha^{\Phi_2}(\mathbb{C}_+)\right)=\{0\}.$$
\end{itemize}
\end{lem}
\begin{proof}
(i)	 Assume that $\omega$ is equivalent to $1$. Then for every
$t  >0$, 	
$$\begin{array}{rcl}
\omega(\frac{1}{t})	\approx 1 &\Rightarrow&  \Phi_{1}^{-1}(t) \approx \Phi_{2}^{-1}(t^{2+\alpha}).
\end{array}$$ 
This means in particular that there exists a constant $C>0$ such that for every $t>0$, $$\Phi_{1}^{-1}(t) \leq \Phi_{2}^{-1}(Ct^{2+\alpha}).$$ 
%\end{array}$$	
As $\Phi_{1}, \Phi_{2}\in \mathcal{U}$, $\Phi_{1}$ satisfies the $\nabla_2$-condition and $\dfrac{\Phi_{2}}{\Phi_{1}}$  nondecreasing, we have by Theorem \ref{thm:embed1}, that ${H}^{\Phi_{1}}({\mathbb{C_{+}}})$ embeds continuously into  ${A}^{\Phi_{2}}_{\alpha}({\mathbb{C_{+}}})$. Thus there is a constant $C>0$ such that $\forall~ f \in {H}^{\Phi_{1}}({\mathbb{C_{+}}}), ~f\not=0$,  \Be\label{eq:hardbergembpro} \int_{{\mathbb{C}}_{+}}\Phi_{2}\left( \dfrac{|f(z)|}{C\|f\|_{{H}^{\Phi_{1}}}^{lux}}\right)dV_{\alpha}(z) \leq 1.\Ee
Let us now prove that $\mathcal{M}( {H}^{\Phi_{1}}({\mathbb{C_{+}}}) ,  {A}^{\Phi_{2}}_{\alpha}({\mathbb{C_{+}}})  ) = {H}^{\infty}({\mathbb{C_{+}}})$. \\
Let $g \in 	{H}^{\infty}({\mathbb{C_{+}}})$ and let	$ f \in {H}^{\Phi_{1}}({\mathbb{C_{+}}}),~ f\not=0$. If $g=0$, then there is nothing to prove. Let us then assume that $g\not=0$. Using (\ref{eq:hardbergembpro}), we obtain 
\[ \int_{{\mathbb{C}}_{+}}\Phi_{2}\left( \dfrac{|g(z)f(z)|}{C\|g\|_{\infty}\|f\|_{{H}^{\Phi_{1}}}^{lux}}\right)dV_{\alpha}(z)\leq \int_{{\mathbb{C}}_{+}}\Phi_{2}\left( \dfrac{|f(z)|}{C\|f\|_{{H}^{\Phi_{1}}}^{lux}}\right)dV_{\alpha}(z) \leq 1 .           \]
Thus \[    \|fg\|_{\Phi_{2},\alpha }^{lux} \leq C\|g\|_{\infty}\|f\|_{{H}^{\Phi_{1}}}^{lux} \leq K\|f\|_{{H}^{\Phi_{1}}}^{lux}. \]
It follows that $g\in \mathcal{M}( {H}^{\Phi_{1}}({\mathbb{C_{+}}}) ,  {A}^{\Phi_{2}}_{\alpha}({\mathbb{C_{+}}})  )$ whenever $g\in {H}^{\infty}({\mathbb{C_{+}}})$.
\vskip .2cm
Let us now prove the converse. Let $g \in \mathcal{M}( {H}^{\Phi_{1}}({\mathbb{C_{+}}}) ,  {A}^{\Phi_{2}}_{\alpha}({\mathbb{C_{+}}})  )$. Then there is a constant $C>0$ such that for any $f \in {H}^{\Phi_{1}}({\mathbb{C_{+}}})$, 
%\[  \int_{{\mathbb{C}}_{+}}\Phi_{2}\left( \dfrac{|g(z)f(z)|}{C\|f\|_{{H}^{\Phi_{1}}}^{lux}}\right)dV_{\alpha}(z) \leq 1. \]
%Ainsi, on d\'eduit que $\forall~ f \in \textbf{H}^{\Phi_{1}}(\pmb{\mathbb{C_{+}}}),~ f\not=0$, on a : $fg \in \textbf{A}^{\Phi_{2}}_{\alpha}(\pmb{\mathbb{C_{+}}})$ et de plus 
\[  \|fg\|_{\Phi_{2},\alpha}^{lux} \leq  C\|f\|_{{H}^{\Phi_{1}}}^{lux}.    \]     \\
It follows from this and Lemma \ref{lem:pointwiseberg} that there is a constant $K>0$ such that for any $z=x+iy\in \mathbb{C}_+$, \Be\label{eq:pointwisetotest}   	| f(z)g(z)| \leq K\Phi_{2}^{-1}\left(\frac{1}{y^{2+\alpha}} \right) \|fg\|_{\Phi_{2},\alpha}^{lux}  \leq  KC\Phi_{2}^{-1}\left(\frac{1}{y^{2+\alpha}} \right)\|f\|_{{H}^{\Phi_{1}}}^{lux}. \Ee
Fix $z_0=x_0+iy_0\in \mathbb{C}_+$ and consider the function $f_{z_{0}}$ defined \[ f_{z_o}(\omega)= \Phi^{-1}_{1}\left(\dfrac{1}{y_{0}}\right) \dfrac{y^{2}_{0}}{ (\omega-\overline{z_{0}})^{2}} ,~ \forall~ \omega \in {\mathbb{C_{+}}}.          \]	
We recall with Lemma \ref{lem:testfuncthardyo} that 	$f_{z_0}\in {H}^{\Phi_{1}}({\mathbb{C_{+}}})$ with $\|f_{z_{0}}\|_{{H}^{\Phi_{1}}}^{lux} \leq \pi$. Replacing $f$ by $f_{z_0}$ in (\ref{eq:pointwisetotest}), we obtain that for any $z=x+iy\in \mathbb{C}_+$,
$$\Phi^{-1}_{1}\left(\dfrac{1}{y_{0}}\right) \dfrac{y^{2}_{0}}{ |z-\overline{z_{0}}|^{2}}|g(z)|\leq C\Phi_{2}^{-1}\left(\frac{1}{y^{2+\alpha}}\right )\pi$$
and the constant does not depend on $z$. As this happens for any $z=x+iy\in \mathbb{C}_+$, taking in particular $z=z_0$, we obtain
\[         |g(z_{0})|\leq 4\pi C\frac{\Phi_{2}^{-1}\left(\frac{1}{y_{0}^{2+\alpha}}\right)}{\Phi^{-1}\left(\frac{1}{y_{0}}\right)}= 4\pi C\omega(y_{0})\approx 4\pi C.  \]
Thus \[      |g(z)|\leq 4\pi C \,\,\,\textrm{for any}\,\,\,z=x+iy\in \mathbb{C}_+. \]
% $\|g\|_{\infty} \preceq C_{1} < \infty$ \\
Hence $g \in {H}^{\infty}({\mathbb{C_{+}}})$.
\vskip .3cm
 (ii) Suppose that the function $\omega$  is nondecreasing and $\lim_{t \to 0}\omega(t)=0$.
Let $g$ be a multiplier from $ {H}^{\Phi_{1}}({\mathbb{C_{+}}})$ to ${A}^{\Phi_{2}}_{\alpha}({\mathbb{C_{+}}})$. We obtain as above that there is a constant $C>0$ such that for any $z=x+iy\in {\mathbb{C_{+}}}$,          
\Be\label{eq:pointwisegene}  |g(z)|\leq 4\pi C\omega(y_). \Ee
%As $\omega$  is nondecreasing on $(0 , \infty)$, it follows that \[    |g(z)|\le 4\pi C\omega(t) ,~~ \forall~ 0 < t < y. \]
%Letting $t\rightarrow 0$, we obtain that $g(z)=0$ for any  $z=x+iy\in {\mathbb{C_{+}}}$.
Letting $y\rightarrow 0$, we obtain from our hypothesis on $\omega$ that the right hand side of (\ref{eq:pointwisegene}) goes to $0$. Thus $g(z)=0$ for all $z\in \mathbb{C}_+$.
Hence $$\mathcal{M}( {H}^{\Phi_{1}}({\mathbb{C_{+}}}) , {A}^{\Phi_{2}}_{\alpha}({\mathbb{C_{+}}})  ) = \{ 0 \}.$$ The proof is complete.	

\end{proof}
We next prove the following.
\begin{lem}\label{lem:main32}
Let $\Phi_1\in \mathscr U$ and $\Phi_2\in \tilde{\mathscr U}$. Assume that $\Phi_1$ and $\Phi_2\circ\Phi_1^{-1}$ satisfy the $\nabla_2$-condtion and that $\frac{\Phi_2}{\Phi_1}$ is non-decreasing. Let $\alpha>-1$ and define for $t\in (0,\infty)$, the function
$$\omega(t)=\frac{\Phi_2^{-1}\left(\frac{1}{t^{2+\alpha}}\right)}{\Phi_1^{-1}\left(\frac{1}{t}\right)}.$$
 If $\omega$ is non-increasing on $(0,\infty)$, then $$\mathcal{M}\left(H^{\Phi_1}(\mathbb{C}_+),A_\alpha^{\Phi_2}(\mathbb{C}_+)\right)=H_\omega^\infty(\mathbb{C}_+).$$
\end{lem}
\begin{proof}
That if $g\in \mathcal{M}\left(H^{\Phi_1}(\mathbb{C}_+),A_\alpha^{\Phi_2}(\mathbb{C}_+)\right)$, then $g\in H_\omega^\infty(\mathbb{C}_+)$, follows from (\ref{eq:pointwisegene}). Let us then prove the converse.
\vskip .2cm 
Let $K=\max\{1,2C_1C_2,2C_1C_3\}$ where $C_1$, $C_2$ and $C_3$ are respectively the constants in conditions (\ref{eq:uppertypecondmulti1}), (\ref{eq:uppertypecondmulti2}) and (\ref{eq:uppertypecondmulti3}) in the definition of the class $\tilde{\mathscr U}$.
Using the property (\ref{eq:uppertypecondmulti1}), we first obtain for $C>0$ a constant whose existence has to be proved, 
\Beas
L &:=& \int_{\mathbb{C}_+}\Phi_2\left(\frac{|g(z)||f(z)|}{KC\|g\|_{H_\omega^\infty}\|f\|_{H^{\Phi_1}}^{lux}}\right)dV_\alpha(z)\\ &\leq& \int_{\mathbb{C}_+}\Phi_2\left(\frac{\Phi_2^{-1}\left(\frac{1}{(\Im z)^{2+\alpha}}\right)}{\Phi_1^{-1}\left(\frac{1}{\Im z}\right)}\frac{|f(z)|}{KC\|f\|_{H^{\Phi_1}}^{lux}}\right)dV_\alpha(z)\\ &\leq& C_1\int_{\mathbb{C}_+}\Phi_2\left(\frac{\Phi_2^{-1}\left(\frac{1}{(\Im z)^{2+\alpha}}\right)}{\Phi_1^{-1}\left(\frac{1}{\Im z}\right)}\right)\Phi_2\left(\frac{|f(z)|}{KC\|f\|_{H^{\Phi_1}}^{lux}}\right)dV_\alpha(z)\\ &=& L_1+L_2
\Eeas
where
$$L_1:=C_1\int_{\mathbb{C}_+}\Phi_2\left(\frac{\Phi_2^{-1}\left(\frac{1}{(\Im z)^{2+\alpha}}\right)}{\Phi_1^{-1}\left(\frac{1}{\Im z}\right)}\right)\Phi_2\left(\frac{|f(z)|}{KC\|f\|_{H^{\Phi_1}}^{lux}}\right)\chi_{\{\Im z>1\}}(z)dV_\alpha(z)$$
and 
$$L_2:=C_1\int_{\mathbb{C}_+}\Phi_2\left(\frac{\Phi_2^{-1}\left(\frac{1}{(\Im z)^{2+\alpha}}\right)}{\Phi_1^{-1}\left(\frac{1}{\Im z}\right)}\right)\Phi_2\left(\frac{|f(z)|}{KC\|f\|_{H^{\Phi_1}}^{lux}}\right)\chi_{\{\Im z\le 1\}}(z)dV_\alpha(z).$$
We observe that as the function $\omega$ is nonincreasing, we have that $$\Phi_2^{-1}\left(\frac{1}{t^{2+\alpha}}\right)\le \Phi_1^{-1}\left(\frac{1}{t}\right)\,\,\,\textrm{for any}\,\,\,t\ge 1.$$
Hence using (\ref{eq:uppertypecondmulti3}) and the definition of the constant $K$, we obtain
\Beas L_1 &\leq& C_1C_3\int_{\mathbb{C}_+}\frac{1}{(\Im z)^{2+\alpha}\Phi_2\circ\Phi_1^{-1}\left(\frac{1}{\Im z}\right)}\Phi_2\left(\frac{|f(z)|}{KC\|f\|_{H^{\Phi_1}}^{lux}}\right)dV_\alpha(z)\\ &\leq& \frac{1}{2}\int_{\mathbb{C}_+}\frac{1}{(\Im z)^{2+\alpha}\Phi_2\circ\Phi_1^{-1}\left(\frac{1}{\Im z}\right)}\Phi_2\left(\frac{|f(z)|}{C\|f\|_{H^{\Phi_1}}^{lux}}\right)dV_\alpha(z).
\Eeas
Now let $q\geq 1$ be the upper-type of $\Phi_2$. Using (\ref{eq:uppertypecondmulti2}), we obtain
\Beas L_2 &\leq& C_1C_3\int_{\mathbb{C}_+}\frac{1}{(\Im z)^{2+\alpha}\left(\Phi_1^{-1}\left(\frac{1}{\Im z}\right)\right)^q}\Phi_2\left(\frac{|f(z)|}{KC\|f\|_{H^{\Phi_1}}^{lux}}\right)\chi_{\{\Im z\le 1\}}(z)dV_\alpha(z)\\ &\leq& \frac{1}{2}\int_{\mathbb{C}_+}\frac{1}{(\Im z)^{2+\alpha}\Phi_2\circ\Phi_1^{-1}\left(\frac{1}{\Im z}\right)}\Phi_2\left(\frac{|f(z)|}{C\|f\|_{H^{\Phi_1}}^{lux}}\right)dV_\alpha(z).
\Eeas
It follows that
$$L\le \int_{\mathbb{C}_+}\frac{1}{(\Im z)^{2+\alpha}\Phi_2\circ\Phi_1^{-1}\left(\frac{1}{\Im z}\right)}\Phi_2\left(\frac{|f(z)|}{C\|f\|_{H^{\Phi_1}}^{lux}}\right)dV_\alpha(z).$$
Hence to conclude, we only have to prove the existence of a constant $C>0$ such that 
$$\int_{\mathbb{C}_+}\Phi_2\left(\frac{|f(z)|}{C\|f\|_{\Phi_1,\alpha}^{lux}}\right)d\mu(z)\leq 1$$ where $$d\mu(x+iy)=\frac{dV(x+iy)}{y^2\Phi_2\circ\Phi_1^{-1}\left(\frac{1}{y}\right)}.$$ By  Theorem \ref{thm:main2}, it is enough to prove that $\mu$ is a $\Phi_2\circ \Phi_1^{-1}$-Carleson measure.
\vskip .2cm
Let $I\subset \mathbb{R}$ be a fixed interval. Let $s$ be the lower indice of $\Phi_2\circ\Phi_1^{-1}$. From the comments at the beginning of the Subsection 3.1, we have that $s>1$. Using that the function $t\mapsto \frac{\Phi_2\circ\Phi_1^{-1}(t)}{t^s}$ is increasing, we obtain that
\Beas
\mu(Q_I) &=& \int_I\int_0^{|I|}\frac{dxdy}{y^2\Phi_2\circ\Phi_1^{-1}\left(\frac{1}{y}\right)}\\ &=& |I|\sum_{j=0}^\infty\int_{2^{-j-1}|I|}^{2^{-j}|I|}\frac{dy}{y^2\Phi_2\circ\Phi_1^{-1}\left(\frac{1}{y}\right)}\\ &\le& |I|\sum_{j=0}^\infty\frac{1}{(2^{-j-1}|I|)^2\Phi_2\circ\Phi_1^{-1}\left(\frac{1}{2^{-j}|I|}\right)}2^{-j}|I|\\ &\le& \frac{4}{\Phi_2\circ\Phi_1^{-1}\left(\frac{1}{|I|}\right)}\sum_{j=0}^\infty 2^{-j(s-1)}\\ &\lesssim& \frac{1}{\Phi_2\circ\Phi_1^{-1}\left(\frac{1}{|I|}\right)}.
\Eeas
The proof is complete.
\end{proof}
\begin{xrem}
For the measure $d\mu(x+iy)=\frac{dV(x+iy)}{y^2\Phi_2\circ\Phi_1^{-1}\left(\frac{1}{y}\right)}$ to be a $\Phi_2\circ\Phi_1^{-1}$-Carleson measure, that $\Phi_2\circ\Phi_1^{-1}$ satisfies the $\nabla_2$-Condition is relevant in our proof. Indeed, if we take $\Phi_1(t)=t^2$ and $\Phi_2(t)=t^2\ln(C+t)$ with $C>0$ large enough, then these two functions are in $\mathscr U$ and obviously, $\Phi_1$ satisfies the $\nabla_2$-condition while $\Phi_2\circ\Phi_1^{-1}(t)=t\ln(C+t^{\frac 12})$ does not, moreover, we have that $\mu$ is not a $\Phi_2\circ\Phi_1^{-1}$-Carleson measure in this case. Indeed, we have for any finite interval $I$,
\Beas
\mu(Q_I) &=& \int_I\int_0^{|I|}\frac{dxdy}{y\ln\left(C+\frac{1}{y^{\frac 12}}\right)}\\ &=& 2|I|\int_{\frac{1}{\sqrt{|I|}}}^\infty \frac{ds}{s\ln(C+s)}\\ &\ge& |I|\int_{\frac{1}{\sqrt{|I|}}}^\infty \frac{ds}{(C+s)\ln(C+s)}\\ &=& |I|\lim_{R\rightarrow \infty}\left[\ln\ln(C+R)-\ln\ln\left(C+\frac{1}{\sqrt{|I|}}\right)\right]\\ &=&\infty.
\Eeas

\end{xrem}
The proof of the following lemma is obtained as for Lemma \ref{lem:main31}.
\begin{lem}\label{lem:main41}
Let $\Phi_1,\Phi_2\in \mathscr U$. Assume that $\frac{\Phi_2}{\Phi_1}$ is non-decreasing. Let $\alpha, \beta>-1$ and define for $t\in (0,\infty)$, the function
$$\omega(t)=\frac{\Phi_2^{-1}\left(\frac{1}{t^{2+\beta}}\right)}{\Phi_1^{-1}\left(\frac{1}{t^{2+\alpha}}\right)}.$$
Then the following assertions hold.
\begin{itemize}
\item[(i)] If $\Phi_1$ satisfies the $\nabla_2$-condition, and $\omega$ is equivalent to $1$, then $$\mathcal{M}\left(A_\alpha^{\Phi_1}(\mathbb{C}_+),A_\beta^{\Phi_2}(\mathbb{C}_+)\right)=H^\infty(\mathbb{C}_+).$$
\item[(ii)] If $\omega$ is nondecreasing and $\lim_{t\rightarrow 0}\omega(t)=0$, then $$\mathcal{M}\left(A_\alpha^{\Phi_1}(\mathbb{C}_+),A_\beta^{\Phi_2}(\mathbb{C}_+)\right)=\{0\}.$$
\end{itemize}
\end{lem}
Let us prove the following.
\begin{lem}\label{lem:main42}
Let $\Phi_1\in \mathscr U$ and $\Phi_2\in  \tilde{\mathscr U}$. Assume that $\Phi_1$ and $\Phi_2\circ\Phi_1^{-1}$ satisfy the $\nabla_2$-condition, and $\frac{\Phi_2}{\Phi_1}$ is non-decreasing. Let $\alpha, \beta>-1$ and define for $t\in (0,\infty)$, the function
$$\omega(t)=\frac{\Phi_2^{-1}\left(\frac{1}{t^{2+\beta}}\right)}{\Phi_1^{-1}\left(\frac{1}{t^{2+\alpha}}\right)}.$$

If $\omega$ is non-increasing on $(0,\infty)$, then $$\mathcal{M}\left(A_\alpha^{\Phi_1}(\mathbb{C}_+),A_\beta^{\Phi_2}(\mathbb{C}_+)\right)=H_\omega^\infty(\mathbb{C}_+).$$
\end{lem}
\begin{proof}
Let $g \in \mathcal{M}( {A}^{\Phi_{1}}_{\alpha}({\mathbb{C_{+}}}),  {A}^{\Phi_{2}}_{\beta}({\mathbb{C_{+}}})  )$. Then using Lemma \ref{lem:pointwiseberg}, and the test function given in Lemma \ref{lem:testfunctbergo}, we obtain as in (\ref{eq:pointwisegene}) that there is a constant $C>0$ such that for any $z=x+iy \in {\mathbb{C_{+}}}$,
\[         |g(z)|\leq C\dfrac{\Phi_{2}^{-1}\left(\frac{1}{y^{2+\beta}}\right)}{\Phi_1^{-1}\left(\frac{1}{y^{2+\alpha}}\right)}= C\omega(y).  \]
Hence \[  \dfrac{|g(z)|}{\omega(y)} \leq C<\infty .          \]
Thus $g \in {H}^{\infty}_{\omega}({\mathbb{C_{+}}})$. 
\vskip .2cm
For the converse, we start by observing that as in the proof of Lemma \ref{lem:main32}, one has that the measure $$ d\mu(x+iy)=\dfrac{dV(x+iy)}{y^{2}\Phi_{2} \circ \Phi_{1}^{-1}(\frac{1}{y^{2+\alpha}})}$$	

is a $(\Phi_2\circ\Phi_1^{-1},\alpha)$-Carleson measure. Hence by Theorem \ref{thm:main2}, the is a constant $C>0$ such that for any $f \in {A}^{\Phi_{1}}_{\alpha}({\mathbb{C_{+}}}), f\not=0$,
$$\int_{{\mathbb{C}}_{+}}\Phi_{2}\left( \dfrac{|f(z)|}{C\|f\|_{\Phi_{1},\alpha} ^{lux}}\right)d\mu(z)\le 1.$$

Let $f \in {A}^{\Phi_{1}}_{\alpha}({\mathbb{C_{+}}}), f\not=0$ , and define	
\[ L_{1}=C_{1}\int_{{\mathbb{C_{+}}}}\Phi_{2}\left(\frac{\Phi_{2}^{-1}\left(\frac{1}{(\Im z)^{2+\beta}}\right)}{\Phi_{1}^{-1}\left(\frac{1}{(\Im z)^{2+\alpha}}\right)}\right)\Phi_{2}\left( \dfrac{|f(z)|}{KC\|f\|_{\Phi_{1},\alpha} ^{lux}}\right)\chi_{\{\Im z > 1\}}(z) dV_{\beta}(z) \]	
and	
\[  L_{2}=C_{1}\int_{{\mathbb{C_{+}}}}\Phi_{2}\left(\frac{\Phi_{2}^{-1}\left(\frac{1}{(\Im z)^{2+\beta}}\right)}{\Phi_{1}^{-1}\left(\frac{1}{(\Im z)^{2+\alpha}}\right)}\right)\Phi_{2}\left( \dfrac{|f(z)|}{KC\|f\|_{\Phi_{1},\alpha} ^{lux}}\right)\chi_{\{\Im z \leq 1\}}(z) dV_{\beta}(z)\]	
where $K=\max\{1, 2C_{1}C_{3} , 2C_{1}C_{2}C_{4} \}$ with $C_{1} , C_{2}, C_{3}$ and $C_{4}$ the constants $(13) , (14), (15)\,\,\textrm{and}\,\,(17)$ respectively.
\vskip .1cm
As $\omega$ is nonincreasing on $(0, \infty)$, we have that $\forall ~t\geq 1$, \[ \Phi_{2}^{-1}\left(\frac{1}{t^{2+\beta}}\right) \leq \Phi_{1}^{-1}\left(\frac{1}{t^{2+\alpha}}\right) \leq 1.\]	
Hence using (15), we obtain 
$$\begin{array}{rcl}
L_{1} &\leq& C_{1}C_{3}\int_{{\mathbb{C_{+}}}}\frac{1}{(\Im z)^{2+\beta}\Phi_{2} \circ \Phi_{1}^{-1}\left(\frac{1}{(\Im z)^{2+\alpha}}\right)}\times\\ && \Phi_{2}\left( \dfrac{|f(z)|}{KC\|f\|_{\Phi_{1},\alpha} ^{lux}}\right)\chi_{\{\Im z > 1\}}(z) dV_{\beta}(z) \\\\	
&\leq& C_{1}C_{3}\int_{{\mathbb{C_{+}}}}\frac{1}{(\Im z)^{2+\beta}\Phi_{2} \circ \Phi_{1}^{-1}\left(\frac{1}{(\Im z)^{2+\alpha}}\right)}\Phi_{2}\left( \dfrac{|f(z)|}{KC\|f\|_{\Phi_{1},\alpha} ^{lux}}\right)dV_{\beta}(z) \\\\	
&\leq& \frac{1}{2}\int_{{\mathbb{C}}_{+}}\Phi_{2}\left( \dfrac{|f(z)|}{C\|f\|_{\Phi_{1},\alpha} ^{lux}}\right)d\mu(z) \\\\	
&\leq& \frac{1}{2}.	
\end{array}$$	
Also, we have that ~ $\forall ~t \leq 1$, \[  \Phi_{1}^{-1}\left(\frac{1}{t^{2+\alpha}}\right) \geq 1 ~~~  \textrm{and}~~~  \Phi_{2}^{-1}\left(\frac{1}{t^{2+\beta}}\right) \geq  1.           \]
Thus if $q\geq 1$ is the upper-type of $\Phi_{2}$, we obtain using (14) that
$$\begin{array}{rcl}
L_{2}
&\leq& C_{1}C_{2}\int_{{\mathbb{C_{+}}}}\frac{1}{(\Im z)^{2+\beta}\left(\Phi_{1}^{-1}\left(\frac{1}{(\Im z)^{2+\alpha}}\right)\right)^{q}}\times\\ & & \Phi_{2}\left( \dfrac{|f(z)|}{KC\|f\|_{\Phi_{1},\alpha} ^{lux}}\right)\chi_{\{\Im z \leq 1\}}(z) dV_{\beta}(z) \\\\
&\leq& C_{1}C_{2}C_{4}\int_{{\mathbb{C_{+}}}}\frac{1}{(\Im z)^{2+\beta}\Phi_{2} \circ \Phi_{1}^{-1}\left(\frac{1}{(\Im z)^{2+\alpha}}\right)}\times\\ & & \Phi_{2}\left( \dfrac{|f(z)|}{KC\|f\|_{\Phi_{1},\alpha} ^{lux}}\right)\chi_{\{\Im z \le 1\}}(z) dV_{\beta}(z) \\\\	
&\leq& C_{1}C_{2}C_{4}\int_{{\mathbb{C_{+}}}}\frac{1}{(\Im z)^{2+\beta}\Phi_{2} \circ \Phi_{1}^{-1}\left(\frac{1}{(\Im z)^{2+\alpha}}\right)}\Phi_{2}\left( \dfrac{|f(z)|}{KC\|f\|_{\Phi_{1},\alpha} ^{lux}}\right)dV_{\beta}(z) \\\\	
&\leq& \frac{1}{2}\int_{{\mathbb{C}}_{+}}\Phi_{2}\left( \dfrac{|f(z)|}{C\|f\|_{\Phi_{1},\alpha} ^{lux}}\right)d\mu(z) \\\\	
&\leq& \frac{1}{2}.	
\end{array}$$  
Now suppose that $g \in {H}^{\infty}_{\omega}({\mathbb{C_{+}}})$. Let us prove that $g \in \mathcal{M}( {A}^{\Phi_{1}}_{\alpha}({\mathbb{C_{+}}}),  {A}^{\Phi_{2}}_{\beta}({\mathbb{C_{+}}})  )$. If $g=0$, then there is nothing to prove, so let us assume that $g\neq 0$.
$\forall~ f \in A_\alpha^{\Phi_1}(\mathbb{C}_+),~ f\not=0$, using the above observations and (13), we obtain 
$$\begin{array}{rcl}
L &:=& \int_{{\mathbb{C}}_{+}}\Phi_{2}\left( \dfrac{|g(z)f(z)|}{KC\|g\|_{\omega}^{\infty}\|f\|_{\Phi_{1},\alpha} ^{lux}}\right)dV_{\beta}(z)\\ 
&\leq &\int_{{\mathbb{C}}_{+}}\Phi_{2}\left(\frac{\Phi_{2}^{-1}(\frac{1}{(\Im z)^{2+\beta}})}{\Phi_{1}^{-1}(\frac{1}{(\Im z)^{2+\alpha}})} \frac{|f(z)|}{KC\|f\|_{\Phi_{1},\alpha} ^{lux}}\right)dV_{\beta}(z) \\\\	
&\leq & C_{1}\int_{{\mathbb{C_{+}}}}\Phi_{2}\left(\frac{\Phi_{2}^{-1}(\frac{1}{(\Im z)^{2+\beta}})}{\Phi_{1}^{-1}(\frac{1}{(\Im z)^{2+\alpha}})}\right)\Phi_{2}\left( \dfrac{|f(z)|}{KC\|f\|_{\Phi_{1},\alpha} ^{lux}}\right) dV_{\beta}(z) \\\\	
&\leq& L_{1} +L_{2} \\\\	
&\leq& 1.   	
\end{array}$$	
Thus $g \in \mathcal{M}( {A}^{\Phi_{1}}_{\alpha}({\mathbb{C_{+}}}),  {A}^{\Phi_{2}}_{\beta}({\mathbb{C_{+}}})  )$  and the proof is complete.
\end{proof}
\section{Further results and concluding remarks}
In this paper, we have presented Carleson embeddings for both Hardy-Orlicz spaces and Bergman-Orlicz spaces, extending the corresponding results for power functions. We have seen with our examples of applications, how useful these embeddings are to understand some other questions of complex analysis and harmonic analysis.
\vskip .1cm
It is possible to obtain weak versions of the above Carleson embeddings using essentially the ideas developed in this paper. Let us start this further discussion by recall that for $\Phi$ a growth function, the weak Orlicz space $L^{\Phi,\infty}(\mathbb{C}_+,\mu)$ consists of all functions $f$ such that $$\|f\|_{\Phi,\infty}:=\sup_{\lambda>0}\Phi(\lambda)\mu\left(\left\{z\in \mathbb{C}_+:\,|f(z)|>\lambda \right\}\right)<\infty.$$  
The characterization of the positive measures $\mu$ such that $H^{1}(\mathbb{C}_+)$ embeds continuously into $L^{1,\infty}(\mathbb{C}_+,\mu)$ is also due to L. Carleson (see \cite{carleson2}). The following is an extension of his result.
\begin{thm}\label{thm:Hardyweak}
Let $\Phi_1$ and $\Phi_2$ be two $\mathcal{C}^1$ convex growth functions with $\Phi_2\in\mathscr{U}$. Assume that $\Phi_1$ satisfies the $\nabla_2$-condition and that $\frac{\Phi_2}{\Phi_1}$ is nondecreasing. Let $\mu$ be a positive Borel measure on $\mathbb{C}_+$. Then the following assertions are equivalent.
\begin{itemize}
\item[(a)] There exists a constant $C_1>0$ such that for any interval $I\subset \mathbb{R}$,
\Be\label{eq:weakhardy1}
 \mu(Q_I)\le \frac{C_1}{\Phi_2\circ\Phi_1^{-1}\left(\frac{1}{|I|}\right)}.
 \Ee
\item[(b)] There exists a constant $C_2>0$ such that for any $f\in H^{\Phi_1}(\mathbb{C}_+)$, $f\neq 0$,
\Be\label{eq:weakhardy2}
\sup_{\lambda>0}\Phi_2(\lambda)\mu\left(\left\{z\in \mathbb{C}_+:\,|f(z)|>C_2\lambda\|f^\star\|_{\Phi_1}^{lux} \right\}\right)\le 1.
%\sup_{x+iy\in \mathbb{C}_+}\int_{\mathbb{C}_+}\Phi_2\left(\Phi_1^{-1}\left(\frac{1}{y}\right)\frac{y^2}{|z-\bar{w}|^2}\right)d\mu(w)\le C<\infty.
\Ee
%where $C>0$ is the constant in (\ref{eq:weakhardy1} ).
\end{itemize}
\end{thm}
\begin{proof}
Assume that (\ref{eq:weakhardy1}) holds.  Then by Lemma \ref{lem:main11} we have that for $f\in H^{\Phi_1}(\mathbb{C}_+)$, $f\neq 0$, and any $\lambda>0$,
$$\mu\left(\left\{z\in \mathbb{C}_+:\,\frac{|f(z)|}{K\|f^\star\|_{\Phi_1}^{lux}}>\lambda \right\}\right)\le C_1\Phi_3\left(\left|\left\{ x\in {\mathbb{R}} : \frac{f^{\star}(x)}{K\|f^\star\|_{\Phi_1}^{lux}} >  \lambda \right\}\right|\right)$$
where $\Phi_3(t)=\frac{1}{\Phi_2\circ\Phi_1^{-1}\left(\frac 1t\right)}$, and $C_1$ is the constant in (\ref{eq:weakhardy1}). We can assume that $C_1>1$, and we define $$E_\lambda:=\left\{ x\in {\mathbb{R}} : \frac{f^{\star}(x)}{\|f^\star\|_{\Phi_1}^{lux}} >  \lambda \right\}.$$ It follows that 
\Beas
S &:=& \Phi_2(\lambda)\mu\left(\left\{z\in \mathbb{C}_+:\,\frac{|f(z)|}{C_1\|f^\star\|_{\Phi_1}^{lux}}>\lambda \right\}\right)\\ &\le& C_1\Phi_2(\lambda)\Phi_3\left(\left|\left\{ x\in {\mathbb{R}} : \frac{f^{\star}(x)}{C_1\|f^\star\|_{\Phi_1}^{lux}} >  \lambda \right\}\right|\right)\\ &\le& \Phi_2(\lambda)\frac{\Phi_3\left(|E_\lambda|\right)}{|E_\lambda|}|E_\lambda|\\ &\le& \Phi_2(\lambda)\frac{\Phi_3\left(\frac{1}{\Phi_1(\lambda)}\right)}{\frac{1}{\Phi_1(\lambda)}}|E_\lambda|\\ &\le& \Phi_1(\lambda)|E_\lambda|\\ &\le& \int_{\mathbb{R}}\Phi_1\left(\frac{f^{\star}(x)}{\|f^\star\|_{\Phi_1}^{lux}}\right)dx\le 1.
\Eeas
Thus (\ref{eq:weakhardy2}) holds.
\vskip .3cm
Let us now assume that (\ref{eq:weakhardy2}) holds. Let
$I\subset {\mathbb{R}}$ be a finite interval and $Q_{I}$ its associated Carleson square. We assume that $Q_I$ is centered at $z_0=x_0+iy_0\in \mathbb{C}_+$. Then by Lemma \ref{lem:testfuncthardyo}, the function $f_0(w):=\Phi_1^{-1}\left(\frac 1{y_0}\right)\frac{y^2}{(w-\bar{z}_0)^2}$ belongs to $H^{\Phi_1}(\mathbb{C}_+)$ and $\|f\|_{H^{\Phi_1}}^{lux}\le \pi$. Also, we have seen that
 $\forall~ w \in Q_{I} ,~ |f_0(w)| > \frac{1}{10}\Phi_1^{-1}\left(\frac{1}{|I|}\right) $. Hence 
\[     Q_{I} \subset \left\{ z\in {\mathbb{C_{+}}} : |f_0(z)| >  \frac{1}{10}\Phi_1^{-1}\left(\frac{1}{|I|}\right) \right\}.\] 
%\Phi_2\circ\Phi_1^{-1}\left(\frac{1}{|I|}\right)
%We can assume that the constant $C:=C_2$ in (\ref{eq:weakhardy2}) is such that $C\ge 1$.
 Then putting $$E_I:=\left\{ w\in {\mathbb{C_{+}}} : \frac{|f_0(w)|}{C_2\|f_0\|_{H^{\Phi_1}}^{lux}} > \frac{1}{10\pi C_2}\Phi_1^{-1}\left(\frac{1}{|I|}\right)  \right\},$$
it follows from our hypothesis that 
\Beas
\Phi_2\left(\Phi_1^{-1}\left(\frac{1}{|I|}\right)\right)\mu( Q_{I} )&\le& C\Phi_2\left(\frac{1}{10\pi C_2}\Phi_1^{-1}\left(\frac{1}{|I|}\right)\right)\mu( Q_{I} )\\ &\leq& C\Phi_2\left(\frac{1}{10\pi C_2}\Phi_1^{-1}\left(\frac{1}{|I|}\right)\right)\mu\left(E_I\right)\\	
&\leq&  C.
\Eeas
Thus $\mu$ is a $\Phi_2\circ\Phi_1^{-1}$-Carleson measure. The proof is complete.

\end{proof}
\vskip .1cm

Similarly, we have the following weak-Carleson embedding result for weighted Bergman-Orlicz spaces.
\begin{thm}\label{thm:Bergweak}
Let $\Phi_1$ and $\Phi_2$ be two growth functions in $\mathscr{U}$. Assume that $\Phi_1$ satisfies the $\nabla_2$-condition and that $\frac{\Phi_2}{\Phi_1}$ is nondecreasing. Let $\mu$ be a positive Borel measure on $\mathbb{C}_+$ and let $\alpha>-1$. Then the following assertions are equivalent.
\begin{itemize}
\item[(a)] There exists a constant $C_1>0$ such that for any interval $I\subset \mathbb{R}$,
\Be\label{eq:bergweak1}
 \mu(Q_I)\le \frac{C_1}{\Phi_2\circ\Phi_1^{-1}\left(\frac{1}{|I|^{2+\alpha}}\right)}.
 \Ee
\item[(b)] There exists a constant $C_2>0$ such that for any $f\in A_\alpha^{\Phi_1}(\mathbb{C}_+)$, $f\neq 0$,
\Be\label{eq:bergweak2}
\sup_{\lambda>0}\Phi_2(\lambda)\mu\left(\left\{z\in \mathbb{C}_+:\,|f(z)|>C_2\lambda\|f\|_{\Phi_1,\alpha}^{lux} \right\}\right)\le 1.
%\sup_{x+iy\in \mathbb{C}_+}\int_{\mathbb{C}_+}\Phi_2\left(\Phi_1^{-1}\left(\frac{1}{y^{2+\alpha}}\right)\frac{y^{4+2\alpha}}{|z-\bar{w}|^{4+2\alpha}}\right)d\mu(w)\le C<\infty.
\Ee
\end{itemize}
\end{thm}

Finally, we remark that in the case of Bergman-Orlicz spaces, one could have also considered a characterization of their Carleson measures in terms of Bergman metric balls. The case of the continuous embeddings $$H^{\Phi_1},A_\alpha^{\Phi_1}\hookrightarrow L^{\Phi_2}(d\mu)$$ for $\frac{\Phi_2}{\Phi_1}$ nonincreasing is still open and is expected to be particularly hard for the case of Hardy-Orlicz spaces.

%\subsection*{Acknowledgements}
%This research was partly supported by NSF (grant no. XXXX).

\end{document}